\documentclass[11 pt]{article}
\usepackage{amsmath}
\usepackage{graphicx}
\usepackage{amsfonts}
\usepackage{amssymb}
\usepackage{MnSymbol}
\usepackage{amsthm}
\usepackage[left=3cm, right=3cm, top=3cm, bottom=3.5cm]{geometry}
\usepackage{indentfirst}
\usepackage[all]{xy}
\usepackage{appendix}
\usepackage[colorlinks=true,linkcolor=blue]{hyperref}
\usepackage{mathrsfs} 
\usepackage{ytableau}
\usepackage{tikz-cd}
\usepackage{graphicx,caption,subcaption}
\usepackage{enumitem}
\setitemize{noitemsep,topsep=0pt,parsep=0pt,
partopsep=0pt,itemindent=12pt,leftmargin=0pt}

\theoremstyle{plain}
\newtheorem{theorem}{Theorem}[section]
\newtheorem{conjecture}{Conjeture}[section]
\newtheorem{lemma}[theorem]{Lemma}
\newtheorem{corollary}[theorem]{Corollary}
\newtheorem{proposition}[theorem]{Proposition}

\theoremstyle{remark}
\newtheorem{remark}{Remark}[section]

\usepackage{sectsty}
\sectionfont{\centering}

\makeatletter

\newcommand{\Rmnum}[1]{\expandafter\@slowromancap\romannumeral #1@}
\makeatother
\def\P{\partial}

\def\ri{{\rm i}}
\def\rd{{\rm d}}

\def\rrw{\rightarrow}

\def\br{\mathbb R}
\def\bn{\mathbb N}
\def\bz{\mathbb Z}
\def\bc{{\mathbb C}}

\def\rrw{\rightarrow}
\numberwithin{equation}{section}
\allowdisplaybreaks[2] 
\providecommand{\keywords}[1]
{
  \small
  \textbf{\textit{Keywords and phrases---}} #1
  \normalsize
}
\providecommand{\MSC}[1]
{
  \small
  \textbf{\textit{Mathematics Subject Classification 2020---}} #1
  \normalsize
}
\title{\LARGE Eventual log-concavity of $k$-rank statistics for \\ integer partitions}
\author{Nian Hong Zhou\footnote{This work was partially supported by Guangxi Science and Technology Plan Project \#2020AC19236, and Guangxi Normal University scientific research startup foundation.}}

\date{}

\begin{document}

\maketitle

\begin{center}
    \vspace*{-0.5cm}
\emph{Dedicated to the memory of Freeman John Dyson, 1923--2020}
    \vspace*{0.5cm}
\end{center}

\begin{abstract}Let $N_k(m,n)$ denote the number of partitions of $n$ with Garvan $k$-rank $m$.
It is well-known that Andrews--Garvan--Dyson's crank and Dyson's rank are the $k$-rank for $k=1$
and $k=2$, respectively.
In this paper, we prove that the sequences
$\left(N_k(m,n)\right)_{|m|\le n-k-71}$ are log-concave for all sufficiently large integers $n$ and each integer $k$. In particular, we partially solve the log-concavity conjecture for Andrews--Garvan--Dyson's crank and Dyson's rank, which was independently proposed by Bringmann--Jennings-Shaffer--Mahlburg and Ji--Zang recently.
\end{abstract}

\keywords{Partitions; $k$-rank; Log-concavity; Uniform asymptotics}

\MSC{Primary 11P82; Secondary 05A16; 05A17}

\maketitle


\section{Introduction and statement of results}
\subsection{Background}\label{sec11}
A \emph{partition} $\lambda:=(\lambda_1,\lambda_2,\ldots,\lambda_r)$ is a finite non-increasing
sequence of positive integers $\lambda_1, \lambda_2, \ldots,  \lambda_r$. The $\lambda_j$ are called the \emph{parts} of the partition. For a partition $\lambda$, let $\#(\lambda)$ denote the number of parts of $\lambda$ and $|\lambda|$ denote the sum of the parts of $\lambda$ with the convention $\#(\varnothing)=|\varnothing| = 0$ for the empty partition $\varnothing$, of $0$.  Furthermore, let $l(\lambda)$ denote the largest part of $\lambda$, $\omega(\lambda)$ denote the number of $1$'s in $\lambda$, and $\mu(\lambda)$ denote the number of parts of $\lambda$ larger than $\omega(\lambda)$. We say $\lambda$ is a partition of $n$ if $|\lambda|=n$. The rank ${\rm rk}(\lambda)$ is defined by
$${\rm rk}(\lambda):=l(\lambda)-\#(\lambda),$$
and the crank ${\rm crk}(\lambda)$ is defined by
$${\rm crk}(\lambda):=\begin{cases}
\qquad l(\lambda)\;\qquad  &\text{if}~\omega(\lambda)=0,\\
\mu(\lambda)-\omega(\lambda) &\text{if}~\omega(\lambda)>0.
\end{cases}$$

\medskip

Let $p(n)$ denote the number of partitions of $n$. The \emph{rank} statistic for integer partitions was introduced by Dyson \cite{MR3077150} in 1944 to explain the congruences of Ramanujan \cite{MR2280868}:
$$
p(5n+4)\equiv 0\pmod 5,
$$
$$
p(7n+5)\equiv 0\pmod 7,
$$
$$
p(11n+6)\equiv 0\pmod {11},
$$
where $n$ is any nonnegative integer.  Dyson conjectured that the rank can give a combinatorial explanation for Ramanujan's congruence
modulo $5$ and $7$.  This was later proven by Atkin and Swinnerton-Dyer \cite{AtkinSD}. However, the rank fails to explain Ramanujan's congruence modulo $11$.  Therefore Dyson \cite{MR3077150} conjectured the existence of another statistic that he called the \emph{crank} which would explain the final Ramanujan congruence. Andrews and Garvan \cite{MR929094, MR920146} successfully found the crank, and proved that the crank simultaneously explains the above three Ramanujan congruences.

\medskip

Let $M(m, n)$ (with a slight modification
in the case that $n = 1$, where the values are instead $M(\pm 1, 1) = 1, M(0, 1) = -1$) and $N(m,n)$ denote the number of partitions of $n$ with crank $m$ and rank $m$, respectively. It is well-known that
$$
\sum_{n\ge 0}M(m,n)q^n:=\frac{1}{(q;q)_{\infty}}\sum_{n\ge 1}(-1)^{n-1}q^{n(n-1)/2+|m|n}(1-q^n),
$$
and
$$
\sum_{n\ge 0}N(m,n)q^n:=\frac{1}{(q;q)_{\infty}}\sum_{n\ge 1}(-1)^{n-1}q^{n(3n-1)/2+|m|n}(1-q^n),
$$
where $(q;q)_\infty=\prod_{j\ge 1}(1-q^j)$. In view of the above expansions, Garvan \cite{MR1291125} generalized Dyson's rank as the following. Let $k\in\bn$ and let $N_k(m,n)$ be defined as
\begin{equation}\label{eqm0}
\mathcal{N}_{k,m}(q):=\sum_{n\ge 0}N_{k}(m,n)q^n=\frac{1}{(q;q)_{\infty}}\sum_{n\ge 1}(-1)^{n-1}q^{n((2k-1)n-1)/2+|m|n}(1-q^n),
\end{equation}
for all $m\in\bz$ and $n\in\bn_0$.  Clearly, $N_1(m,n)=M(m,n)$ and $N_2(m,n)=N(m,n)$. For each integer $k\ge 3$, Garvan \cite{MR1291125} proved that $N_k(m, n)$ is the number of partitions of $n$ into at least $(k-1)$ successive Durfee squares with \emph{$k$-rank} equal to $m$. For the detail of the combinatorial interpretation of $N_k(m,n)$, see \cite[Theorem (1.12)]{MR1291125}.

\medskip
One of the well-known results for $p(n)$ is the following Hardy-Ramanujan asymptotic formula \cite{MR1575586}:
\begin{equation*}
p(n)\sim \frac{1}{4\sqrt{3}n}e^{2\pi\sqrt{{n}/{6}}},
\end{equation*}
as $n\rrw +\infty$. In 1989, Dyson \cite{MR1001259} gave the following crank asymptotic formula conjecture:
\begin{equation*}
M\left( m,n \right)\sim \frac{\pi}{4\sqrt{6n}} {\rm sech}^2 \left( \frac{\pi m}{2\sqrt{6n}}   \right) p(n),
\end{equation*}
as $n\rrw +\infty$. Dyson also asked a problem about the precise range of $m$ in which the asymptotic formula holds and about the error term. In \cite[Theorem 1.2]{MR3451872}, Bringmann and Dousse answer all of these questions firstly by proving that the above conjecture holds for all $|m|\le \pi^{-1}\sqrt{n/6}\log n$.  Later, Dousse and Mertens in \cite{MR3337213} proved the above Dyson's crank conjecture also holds for the rank function $N(m,n)$. The author solved the problem about the precise range in \cite{MR3924736}. The complete answer of the above questions of Dyson was given by Liu and the author in the recent work \cite{LZ}. In fact, they \cite[Theorem 1.3]{LZ} established the following uniform asymptotic formula of Garvan $k$-rank functions $N_k(m,n)$ for each integer $k\ge 1$.
\begin{theorem} Let $k\in\bn$. Let $m\in\bz$ and $n\in\bn$ such that $m= o(n^{3/4})$. We have
\begin{align*}
\frac{N_k(m,n)}{p(n)}=\frac{\pi}{4\sqrt{6n}}{\rm sech}^2\left(\frac{\pi m}{2\sqrt{6n}}\right)\left(1+O\left(\frac{n+m^2}{n^{3/2}}\right)\right),
\end{align*}
as $n\rrw +\infty$.
\end{theorem}
Liu and Zhou \cite[Theorem 1.4]{LZ} also established the following asymptotic monotonicity properties of $N_k(m,n)$ for each integer $k\ge 1$.
\begin{theorem}\label{eqm11}Let $k,n\in\bn$. Uniformly for all $m\in\bz$,
$$\frac{N_k(m,n+|m|)-N_k(m+1,n+|m|)}{\pi^2 p(n)/6n}\sim \left(1+e^{-\pi|m|/\sqrt{6n}}\right)^{-2}\tanh\left(\frac{\pi(2m+1)}{4\sqrt{6n}}\right),$$
as $n\rrw +\infty$.
\end{theorem}

Recall that a finite sequence of real numbers $\left(c_k\right)_{k=0}^n$ is said to be unimodal if there exists a $p$ such that
$$c_0\le c_1\le \cdots\le c_{p-1}\le c_p\ge c_{p+1}\ge \cdots\ge c_n.$$
As a consequence of above Theorem \ref{eqm11},  Liu and the author \cite[Corollary 1.5]{LZ}\footnote{There is a typo in the statement of this corollary. The correct statement is as follows.} established the following eventual unimodal properties:
\begin{theorem}\label{thm3}
For each positive integer $k$, the sequences $\left(N_k(m,n)\right)_{|m|\le n-k}$ are unimodal for all sufficiently large positive integers $n$.
\end{theorem}

\subsection{Main results}\label{sec12}
The monotonicity properties of $M(m,n)$ and $N(m,n)$ have been investigated by many authors.
For examples, Chan and Mao \cite[Theorem 2]{MR3190432} proved that $N(m,n)\ge N(m+2, n)$ for all integers $m, n\ge 0$;
and Ji and Zang \cite[Corollary 1.8]{JZ} proved that the sequence $\left(M(m,n)\right)_{|m|\le n-k}$ is unimodal for integer $n\ge 44$.
The unimodality of $N(m, n)$ conjectured by Chan and Mao \cite[Open question 1]{MR3190432} is still an open problem. Ji and Zang \cite[Conjecture 11.4]{JZ}, and in the recent preprint \cite{BGRZ} of Bringmann--Gomez--Rolen--Tripp, the conjecture that the sequence $\left(N(m,n)\right)_{|m|\le n-k}$ is unimodal for integer $n\ge 39$ has been raised. Although Liu and the author \cite[Corollay 1.5]{LZ} (see also the above Theorem \ref{thm3}) proved that the sequences $\left(N_k(m,n)\right)_{|m|\le n-k}$ are unimodal for all sufficiently large integers $n$, the exact lower bound of $n$ remains to be studied.
In view of this, we made the following unimodality conjecture for Garvan $k$-rank functions. We observe that it is true for $k\in\{2,3,4,5,6,7,8,9,10\}$ and $n\le 1000$ (checked by {\bf Mathematica}).
\begin{conjecture}\label{conj1}Let ${\bf 1}_{event}$ be the indicator function.
Define for each integer $k\ge 2$ that
$$n_u(k)=(k+1)+36\cdot {\bf 1}_{k=2}+6\cdot {\bf 1}_{k=3}.$$
Then the sequence $\left(N_k(m,n)\right)_{|m|\le n-k}$ is a unimodal sequence for integer $n\ge n_u(k)$.
\end{conjecture}
\medskip
In this paper, we investigate the log-concavity of sequences $\left(N_k(m,n)\right)_{|m|\le n-k}$. Recall that a finite sequence of real
numbers $\left(c_k\right)_{k=0}^n$ is said to be log-concave if
$$c_i^2-c_{i-1}c_{i+1}\ge 0,$$
holds for every $c_i$ with $1\le i\le n-1$. There are related results for the partition function $p(n)$. For examples, Nicolas \cite{MR513879} and DeSalvo--Pak \cite[Theorem 1.1]{MR3396486} proved that the partition number sequence $\left(p(n)\right)_{n\ge 25}$
is log-concave.  Since a log-concave sequence of positive numbers is also unimodal, we have that the log-concavity of sequences $\left(N_k(m,n)\right)_{|m|\le n-k}$ implies its unimodality.  Recently, Bringmann, Jennings-Shaffer and Mahlburg \cite[Conjecture 4.3]{MR4299082}, and later, Ji and Zang \cite[Conjectures 11.1, 11.2]{JZ}
conjectured that the following log-concavity for both the rank and crank of the integer partitions.
\begin{conjecture}\label{conj2}The following inequalities hold:
$$M(m,n)^2\ge M(m-1,n)M(m+1,n),\qquad  \text{for}~n\ge 71~\text{and}~|m|\le n-71, $$
$$N(m,n)^2\ge N(m-1,n)N(m+1,n),\qquad  \text{for}~n\ge 72~\text{and}~|m|\le n-72. $$
\end{conjecture}
We observe that the same phenomenon appears to occur for all Garvan $k$-rank functions $N_k(m,n)$, which we conjecture as the following Conjecture \ref{conj3}. In particular, the cases of $k=1,2$ are the Conjecture \ref{conj2}.  We observe that it is true for $k\in\{1,2,3,4,5,6,7,8,9,10\}$ and $n\le 1000$ (checked by {\bf Mathematica}).
\begin{conjecture}\label{conj3}
For all integers $k, n\ge 1$ with $n\ge k+71$, the sequences $\left(N_k(m,n)\right)_{|m|\le n-k-71}$ are log-concave.
\end{conjecture}

We prove that Conjecture \ref{conj3} is true for \emph{all sufficiently large $n$} and \emph{each $k\ge 1$}. In other words, for each $k\ge 1$ the sequences $\left(N_k(m,n)\right)_{|m|\le n-k-71}$ are log-concave, eventually holds for all large enough integers $n$. The main results of this paper are as follows.
\begin{theorem}\label{thmlc0}
For each integer $k\ge 1$, there exists a constant $n_{\rm lc}(k)$ such that the sequences $\left(N_k(m,n)\right)_{|m|\le n-k-71}$ are log-concave for all integers $n\ge n_{\rm lc}(k)$.
\end{theorem}
Theorem \ref{thmlc0} follows from the following Theorem \ref{thmlca0} and Proposition \ref{prolcl0}. Throughout the paper, we denote $\beta_\ell=\pi/\sqrt{6(\ell-1/24)}$, and $\Delta_w^2f(w)=f(w+1)-2f(w)+f(w-1)$ to be the second order central difference of function $f(w)$. We establish the following uniform asymptotic formula.
\begin{theorem}\label{thmlca0}Let $k\in\bn$. Uniformly for all integers $m$ and $ n$,
\begin{equation*}
-\Delta_m^2\log N_k(m,n)\sim \frac{\beta_n^2}{2}{\rm sech}^2\left(\frac{m\beta_n}{2}\right)+\frac{3}{\pi^2}\beta_{n-|m|}^3,
\end{equation*}
as $n-|m|\rrw +\infty$.
\end{theorem}
\begin{remark}We remark that:
\begin{enumerate}
  \item As noted by Bringmann and Dousse \cite{MR3451872}, for such kinds of problems, if $m$ is fixed then one can directly obtain asymptotic formulas since the generating function is the product of a modular form and a false/partial theta function, see for examples \cite{MR3210725}, \cite{MR3279269}, \cite{MR4024548} and \cite{MR4299082}. However, our Theorem \ref{thmlca0} is a bivariate uniform asymptotic. Indeed, this fact is the reason why this problem is difficult.
  \item Note that our methods of proof would allow determining further terms in the asymptotic expansion of $-\Delta_m^2\log N_k(m,n)$, see Theorems \ref{main0}, \ref{lnkmmae}, \ref{thm41}.
\end{enumerate}

\end{remark}

We emphasize that a uniform asymptotic expansion of $N_k(m+j,n)$ has been established in the recent work \cite{LZ} of Liu and the author. Theoretically, by using that asymptotic expansions, we can also establish the asymptotic log-concavity of $k$-rank functions. Since the problem considered in \cite{LZ} is more general then the present paper, the expression form and the calculation of the values of the coefficients in the asymptotic expansions of \cite{LZ} are more complicated! It is quite difficult to establish the log-concavity through this expansion. However, the simpler asymptotic expansion in this paper allows us to calculate the log-concavity of $N_k(m,n)$ more easily. More importantly, we expect that our methods in the proof of Theorem \ref{thmlca0} also apply to show the similar eventual unimodality/log-concavity of this paper, for the rank of unimodal
sequences and strongly unimodal sequences, which conjectured by Bringmann--Jennings-Shaffer--Mahlburg \cite[Conjecture 4.1]{MR4299082} and Bringmann--Jennings-Shaffer--Mahlburg--Rhoades \cite[Conjecture 1.4]{MR4024548}.

\medskip

Theorem \ref{thmlca0} solves the case of $n-|m|\rrw +\infty$ in Theorem \ref{thmlc0}. For the proof of Theorem \ref{thmlca0} with $|m|\le \beta_n^{\varepsilon-3/2}$ for any small $\varepsilon>0$, we use Wright's version of the Circle Method \cite{MR1575956, MR282940}, and based on a uniform asymptotic expansions for the partial/false theta functions established in the recent work \cite[Theorems 2.7]{LZ} of Liu and the author.  This is the most difficult part of this paper. While for the proof of Theorem \ref{thmlca0} with $|m|\beta_n\ge 3\log n$ and $n-|m|\rrw+\infty$, we are based on the Hardy-Ramanujan asymptotic formula \cite[Equation (1.66)]{MR1575586} with some technical lemmas. The proof of this part is simpler, but involves novel techniques. The proof of Theorem \ref{thmlc0} for the case of $|m|\ge n/2-2k+2$ easily follows from a recent work of Chen--Jia--Wang \cite{MR3976587} about \textit{higher order Tur\'{a}n inequalities} for $p(n)$. In particular, we prove the following inequality.

\begin{proposition}\label{prolcl0}For every $k\in\bn$ and $n\ge 142+2k$, the sequence $\left(N_k(m,n)\right)_{n/2\le m\le n-k-71}$ is log-concave. In other words,
$$N_k(m,n)^2>N_k(m-1,n)N_k(m+1,n)$$
for all $n/2-2k+2\le |m|\le n-k-71$.
\end{proposition}
We conclude this section with the discussion of some further combinatorial properties of $N_k(m,n)$. Recall that a finite sequence $(c_j)_{j=0}^n$ of real numbers satisfies the higher order Tur\'{a}n inequalities if
$$4(c_j^2-c_{j-1}c_{j+1})(c_{j+1}^2-c_{j}c_{j+2})-(c_{j}c_{j+1}-c_{j-1}c_{j+2})^2\ge 0,$$
for all $1\le j\le n-2$. A real polynomial is said to be \textit{hyperbolic} if all of its zeros are real. The Jensen polynomial of degree $d$
and shift $r\in\bn_0$ associated to $(c_j)_{j=0}^n$ is given by
\begin{equation*}
J_{c}^{d, r}(X):=\sum_{0\le j\le d}\binom{d}{j}c_{r+j}X^j.
\end{equation*}
We note that Griffin--Ono--Rolen--Zagier \cite{MR3963874} show that Jensen polynomials for a large family of real sequence, including those associated to the Taylor coefficients of Riemann $\xi$ function and the partition function $p(n)$, are eventually hyperbolic.
It is an interesting problem that whether the sequence $\left(N_k(m,n)\right)_{|m|\le n-\delta_k}$ (for some $\delta_k>0$)
satisfies the higher order Tur\'{a}n inequalities, and further the eventual hyperbolicity of the associated Jensen polynomials of any given degree $d$. However, numerical experiments (see Figure \ref{fig:sub1}) give a negative answer to this problem. It is not satisfied with the higher order (asymptotically) Tur\'{a}n inequalities. Therefore, the associated Jensen polynomials are \textit{most likely not eventually hyperbolic for any $d\ge 3$}. We believe that Theorem \ref{main0} in this paper will be able to give a rigorous proof of this conclusion, which we leave to interested readers.

\begin{figure}
\begin{subfigure}{.49\columnwidth}
\centering
\includegraphics[width=0.99\columnwidth]{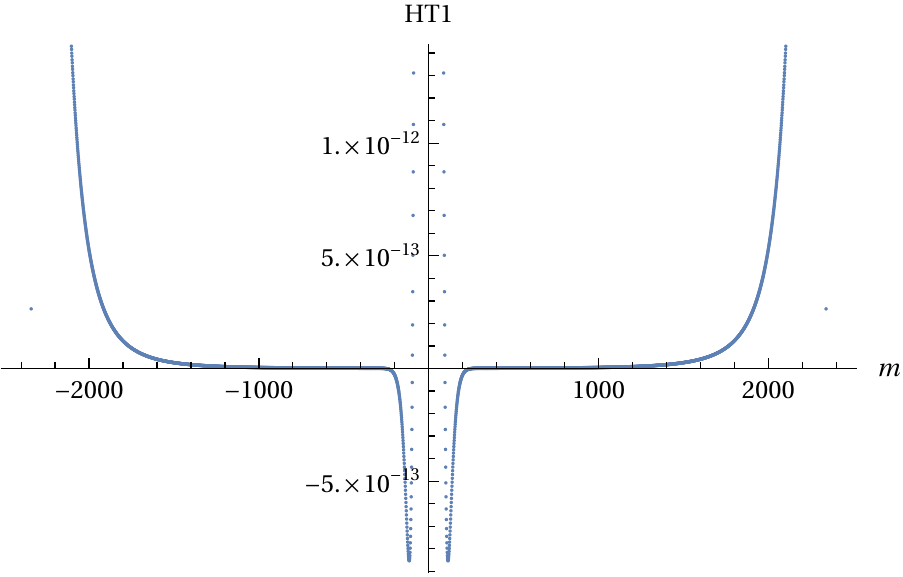}
\caption{Values of ${\rm HT}_1(m,50^2)$}
\label{fig:sub1}
\end{subfigure}
\begin{subfigure}{.49\columnwidth}
\centering
\includegraphics[width=0.99\columnwidth]{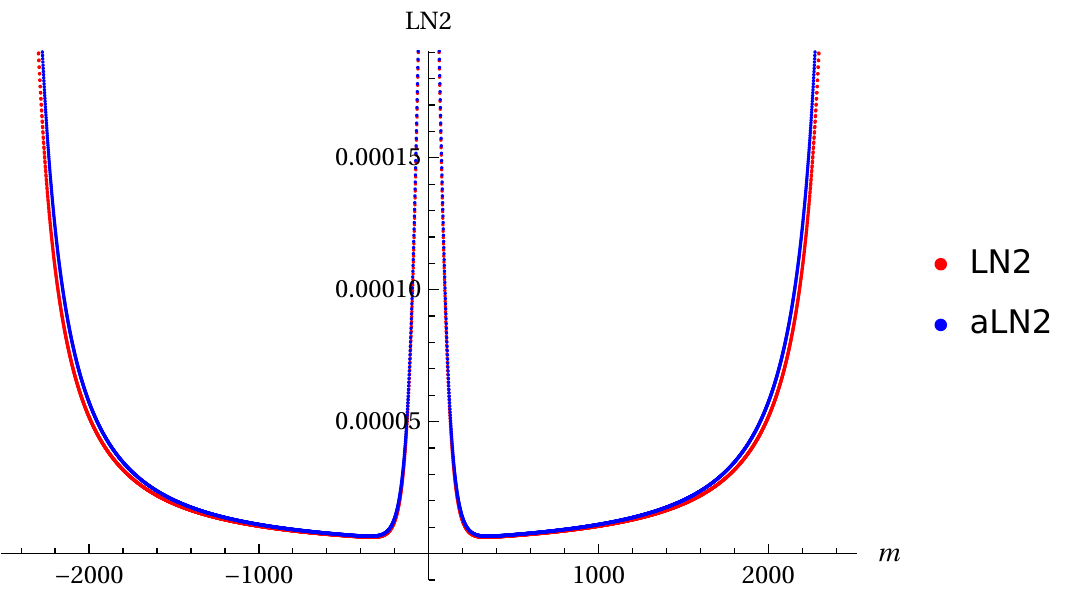}
\caption{Log-concavity of $N_2(m,50^2)$}
\label{fig:sub2}
\end{subfigure}%

 \caption{Data for $N_k(m,n)$}

\medskip

(a) We write $\widetilde{N}_{k}(m, n)=N_k(m,n)^{-2}N_{k}(m-1,n)N_{k}(m+1,n)$, then the higher order Tur\'{a}n inequalities can be restated as
$$
{\rm HT}_k(m,n):=4\left(1-\widetilde{N}_{k}(m,n)\right)\left(1-\widetilde{N}_{k}(m+1,n)\right)-\left(1-\widetilde{N}_{k}(m,n)\widetilde{N}_{k}(m+1,n)\right)^2>0.
$$
From (a) we see that the sign of ${\rm HT1}:={\rm HT}_1(m,50^2)$ changed four times when $|m| \beta_n$ is proportional to $1$.
(b) The ${\rm LN2}:=-\Delta_m^2\log N_2(m,n)$ (red) vs. the asymptotic ${\rm aLN2}:=\frac{\beta_n^2}{2}{\rm sech}^2\left(\frac{m\beta_n}{2}\right)+\frac{3}{\pi^2}\beta_{n-|m|}^3$  (blue) with $n=50^2$ of ${\rm LN2}$. We see that the numerical data supports Theorem \ref{thmlca0}.
\end{figure}

\medskip

The paper is organized as follows. After recalling some preliminaries on asymptotic notations and Bessel functions in Section \ref{sec2}, we prove in Section \ref{sec3} a uniform asymptotic expansion of the generating function $\mathcal{N}_{k, m}(q)$ near $q=1$ with an explicit formula for the occurring error term. In Section \ref{sec4} we use Wright's version of the Circle Method to prove a complete asymptotic expansion for $N_k(m,n)$. In Section \ref{sec5}, we prove Theorem \ref{thmlc0}. We shall prove for the case of $|m|\le \beta_n^{\varepsilon-3/2}$ in Subsection \ref{subsec51}; using the Hardy-Ramanujan asymptotic formula prove for the case of $|m|\beta_n\ge 3\log n$ and $n-|m|\rrw+\infty$ in Subsection \ref{subsec52}; and finally, using a recent work of Chen--Jia--Wang \cite{MR3976587}, we prove the case of $|m|\ge n/2$ in Subsection \ref{subsec53}.

\section{Preliminaries}\label{sec2}
In this section we recall and prove some results required for this paper.
\subsection{Asymptotic notations}

Throughout the paper, we use $f=O(g)$ or $f\ll g$ denotes $|f|\le c g$ for
some constant $c$. If the constant $c$ depend on additional parameters, then we use $f= O_{a,b}(g)$ or $f\ll_{a, b} g$ to denote that $|f| \le c_{a, b} g$ for some constant $c_{a,b}$ depending on $a, b$. If $f, g$ both depend on some parameter $w$, we say that $f=o(g)$ as $w\rightarrow w_0$ if one has $|f|\le c(w)g$
for some function $c(w)$ of $w$, which goes to zero as $w\rightarrow w_0$, for $w_0$ is some fixed parameter.

\medskip

Following \cite{MR0078494}, let $\left(\phi_s(w)\right)_{s\ge 1}$ be a sequence of functions defined on certain subset $\Omega$ of $\bc$ such that for each $s$,
$$\phi_{s+1}(w)=o(\phi_s(w))\;\;\text{as} \;\;w\rrw w_0,$$	
then we call $\left(\phi_s(w)\right)_{s}$ is an asymptotic sequence. Suppose also that $F(w)$ and $f_s(w)$ satisfy
$$F(w)=\sum_{s=0}^{p}f_s(w)+o(\phi_n(w))\;\;\text{as} \;\;w\rrw w_0,$$
for each integer $p\ge 0$. Then we call
$$F(w)\sim\sum_{s\ge 0}f_s(w)\;\;\text{as} \;\;w\rrw w_0,$$
an asymptotic expansion of $F(w)$ with respect to the asymptotic sequence $\left(\phi_s(w)\right)_{s\ge 0}$.

\subsection{Asymptotics involving Bessel functions}
To use the Wright's version of the Circle Method, we require the following results for the Bessel functions. Let $\mu\in\br$ and $u>0$. The modified Bessel functions of first kind $I_{\mu}$ of order $\mu$ is defined by the following integral representation:
$$I_{\mu}(u)=\frac{1}{2\pi\ri }\int_{\mathcal{H}}w^{-\mu-1}e^{(u/2)(w+1/w)}\rd  w,$$
where $\mathcal{H}$ denotes a Hankel contour, which starts from $-\infty$, encircles the origin counter-clockwise, and then returns to $-\infty$.

\medskip

We need the following lemma.
\begin{lemma}\label{eqbesseq1}Let $\delta\in [0,1]$. Then we have
$$I_{\mu}(u)-\frac{1}{2\pi\ri}\int_{1-\ri \delta}^{1+\ri \delta}w^{-\mu-1}e^{(u/2)(w+1/w)}\rd  w\ll_\mu \exp\left(u-\frac{\delta^2u}{4}\right),$$
as $u\rrw +\infty$.
\end{lemma}
\begin{proof}
Using the above integral representation for $I_\mu$, and in view of the Cauchy's theorem, we can choose an integral path such that
\begin{align*}
R_-+R_+=I_{\mu}(u)-\frac{1}{2\pi\ri}\int_{1-\ri \delta}^{1+\ri \delta}w^{-\mu-1}e^{(u/2)(w+1/w)}\rd  w,
\end{align*}
where
\begin{align*}
R_\pm=\pm\frac{1}{2\pi\ri}\left(\int_{1\pm \ri \delta}^{1\pm \ri }+\int_{1\pm \ri }^{-\infty\pm \ri }\right)w^{-\mu-1}e^{(u/2)(w+1/w)}\rd  w.
\end{align*}
Hence, by noting that $\delta \in[0,1]$ and $u>0$, a straightforward calculation gives,
\begin{align*}
R_\pm&\ll \left|\int_{\delta}^{1}(1\pm\ri y)^{-\mu-1}e^{\frac{u}{2}(1\pm\ri y+1/(1\pm\ri y))}\rd  y\right|+\left|\int_{1}^{-\infty}(x\pm \ri )^{-\mu-1}e^{\frac{u}{2}((x\pm \ri )+1/(x\pm \ri ))}\rd  x\right|\\
&\ll_\mu \int_{\delta}^{1}e^{\frac{u}{2}(1+1/(1+y^2))}\rd  y+\int_{-1}^{+\infty}\frac{1}{(x^2+1)^{(\mu+1)/2}}e^{\frac{u}{2}(-x-x/(x^2+1))}\rd  x\\
&\ll_\mu \exp\left(\frac{u}{2}\left(1+\frac{1}{1+\delta^2}\right)\right)+\int_{-1}^{1}e^{\frac{u}{2}(-x-x/(x^2+1))}\rd  x+\int_{1}^{+\infty}x^{|\mu|-1}e^{-\frac{ux}{2}}\rd  x.
\end{align*}
That is
\begin{align*}
R_\pm&\ll_\mu \exp\left(u-\frac{\delta^2u}{2(1+\delta^2)}\right)+\exp\left(\frac{3u}{4}\right)+\int_{1}^{+\infty}x^{|\mu|}e^{-\frac{ux}{2}}\rd  x\\
&\ll_\mu \exp\left(u-\frac{\delta^2u}{4}\right)+\exp\left(u-\frac{u}{4}\right)+u^{-|\mu|-1}\\
&\ll_\mu \exp\left(u-\frac{\delta^2u}{4}\right),
\end{align*}
for $\delta\in[0,1]$ and $u\ge 1$.
\end{proof}
We also noting that the modified Bessel function $I_{\nu}(u)$ have the following asymptotic power series expansion:
\begin{equation}\label{eqbsl}
I_{\mu}(u)\sim \frac{e^u}{\sqrt{2\pi u}}\sum_{j\ge 0}(-1)^ja_j(\mu)\frac{1}{u^{j}},
\end{equation}
as $u\rrw+\infty$, where $a_0(\mu)=1$, and
$$a_j(\nu)=\frac{1}{8^jj!}\prod_{1\le s\le j}((2\mu)^2-(2s-1)^2)$$
is polynomial of $\mu$ with degree $2j$. See \cite[Section 10.40]{NIST:DLMF} for details.

\medskip

In order to simplify the discussion of this paper, we define a new function $H_{\mu,\nu}(u)$ to replace the modified Bessel Functions $I_\mu(u)$.  For $\mu\in\br, \nu\in \bz_{\ge 0}$ and $u>0$, we define
\begin{equation}\label{eqhmv}
H_{\mu,\nu}(u)=\frac{1}{2\pi\ri }\int_{\mathcal{H}}w^{-\mu-1}(z-1)^\nu e^{(u/2)(w+1/w)}\rd  w.
\end{equation}
Clearly, $H_{\mu,\nu}(u)$ is a linear combination of the modified Bessel functions $I_\mu(u)$:
\begin{equation*}
H_{\mu,\nu}(u)=\sum_{0\le h\le \nu}(-1)^h\binom{\nu}{h}I_{\mu+\nu-h}(u).
\end{equation*}
Then, using the fact that
$$
\sum_{0\le h\le \nu}(-1)^h\binom{\nu}{h}h^r
=\begin{cases}\quad 0\;\qquad &r\in[0,\nu)\cap\bz,\\
(-1)^{\nu}\nu !&r=\nu,
\end{cases}
$$
together with \eqref{eqbsl}, we obtain the following asymptotic power series expansion:
\begin{equation}\label{eqm2hv}
H_{\mu,\nu}(u)\sim \frac{e^u}{\sqrt{2\pi u}}\sum_{j\ge \nu/2}\left(\sum_{0\le h\le \nu}(-1)^h\binom{\nu}{h}a_j(\mu+\nu-h)\right)\frac{(-1)^j}{u^{j}},
\end{equation}
as $u\rrw+\infty$. We further define
\begin{equation*}
\widehat{H}_{\mu,v}(u)=\sqrt{2\pi u}e^{-u}H_{\mu,v}(u).
\end{equation*}
The following lemma is a direct consequence of \eqref{eqm2hv}.
\begin{proposition}\label{probh}As $u\rrw+\infty$
$$\widehat{H}_{\mu,\nu}(u)\sim \sum_{\ell\ge \nu/2}\frac{(-1)^{\ell}\gamma_\ell(\mu,\nu)}{u^\ell},$$
where $\gamma_0(\mu,0)=1$,
$$(\gamma_1(\mu, 0), \gamma_1(\mu, 1), \gamma_1(\mu, 2))=\left(\frac{4\mu^2-1}{8},\frac{2\mu+1}{2},1\right),$$
and $\gamma_\ell(\mu,\nu)$ is a polynomial of $\mu$ with degree $2\ell-\nu$ given by \eqref{eqm2hv}, for each pair $(\ell, \nu)\in \bz_{\ge 0}^2$ such that $\ell\ge \nu/2$.
\end{proposition}

\section{Asymptotic expansions for the generating functions}\label{sec3}

 In this section, we establish the uniform asymptotic expansions of the generating function $\mathcal{N}_{k, m} (q) $, which plays a key role in the proof of our main results. We write $q=e^{-z}$, $z=x+\ri y$ with $x>0$ and $y\in\br$. Define
\begin{equation}\label{eqm210}
H_{k,m,j}(q)=\sum_{n\ge 1}(-1)^{n-1}q^{(k-1/2)n^2+mn}(q^{-n/2}-q^{n/2})q^{jn},
\end{equation}
for all $m,j\in\bz$ such that $m+j\ge 0$. Then, the generating function can be rewritten as
\begin{equation}\label{eqmm0}
\mathcal{N}_{k,m+j}(q)=\sum_{n\ge 0}N_{k}(m+j,n)q^n=\frac{1}{(q;q)_{\infty}}H_{k,m,j}(q).
\end{equation}
A parameter $j$ was introduced in the above definition. So that we can conveniently compute the log-concavity of $N_{k}(m,n)$ with respect to $m$.

\medskip

We require the following asymptotic behavior of $1/(q;q)_{\infty}$ which follows directly from the modular transformation of the Dedekind eta function $\eta(\frac{\ri z}{2\pi})=q^{\frac{1}{24}}(q;q)_{\infty}$.
\begin{lemma}\label{lemdek}Let $z=x+\ri y$ with $x,y\in\br$, $x\in(0,1]$. Then for $|y|\ll x^{1/2}$,
$$\frac{1}{(e^{-z};e^{-z})_{\infty}}=\frac{z^{1/2}}{\sqrt{2\pi}}e^{-\frac{z}{24}+\frac{\pi^2}{6z}}+O(|z|^{1/2}),$$
and for $|y|\le \pi$,
$$\frac{1}{(e^{-z};e^{-z})_{\infty}}\ll x^{1/4}\exp\left(\frac{\pi^2}{6 x}\max\left(\frac{1}{1+(y/x)^2}, \frac{1}{4}\right)\right).$$
\end{lemma}
\begin{remark}
We note that some similar results to Lemma \ref{lemdek} has appeared in many related literature, see \cite[Lemma 3.5]{MR3451872} for example. However, the above lemma has smaller error term and smaller upper bound.
\end{remark}

\begin{proof}
Recall the well-known fact for Dedekind eta function that
\begin{equation*}
e^{-\frac{z}{24}}(e^{-z};e^{-z})_{\infty}=\left(\frac{2\pi}{z}\right)^{1/2}e^{-\frac{\pi^2}{6z}}\left(e^{-4\pi^2/z};e^{-4\pi^2/z}\right)_{\infty},
\end{equation*}
and using the definition of partition function $p(\ell)$, we obtain
\begin{align*}
\frac{1}{(e^{-z};e^{-z})_{\infty}}&=\frac{z^{1/2}}{\sqrt{2\pi}}e^{-\frac{z}{24}+\frac{\pi^2}{6z}}+\frac{z^{1/2}}{\sqrt{2\pi}}e^{-\frac{z}{24}}\sum_{\ell\ge 1}p(\ell)e^{-4\pi^2(\ell-\frac{1}{24})\frac{1}{z}}.
\end{align*}
Therefore, by note that $p(\ell)\ll e^{2\pi\sqrt{\ell/6}}$ and
$$\Re\left(\frac{1}{z}\right)=\frac{x}{x^2+y^2}=\frac{1}{x+y^2/x},$$
we have if $1\ll \Re(1/z)$, that is $|y|\ll x^{1/2}$, then
\begin{align}\label{eqdek}
\frac{1}{(e^{-z};e^{-z})_{\infty}}&=\frac{z^{1/2}}{\sqrt{2\pi}}e^{-\frac{z}{24}+\frac{\pi^2}{6z}}+O(|z|^{1/2}),
\end{align}
and if $x^{1/2}\ll |y|\le \pi$, that is $0<\Re(1/z)=x/|z|^2\ll 1$ then,
\begin{align*}
\frac{1}{(e^{-z};e^{-z})_{\infty}}&\ll |z|^{1/2} e^{\frac{\pi^2}{6}\Re\left(\frac{1}{z}\right)}\frac{1}{\left(e^{-4\pi^2\Re(1/z)};e^{-4\pi^2\Re(1/z)}\right)_{\infty}}\\
&\ll |z|^{1/2}(x/|z|^2)^{1/2} e^{\frac{\pi^2}{6}\frac{x}{|z|^2}+\frac{\pi^2}{6}\frac{1}{4\pi^2x/|z|^2}}\frac{1}{(e^{-|z|^2/x};e^{-|z|^2/x})_{\infty}}\\
&\ll(x/|z|)^{1/2}\exp\left(\frac{\pi^2}{6}\frac{x}{|z|^2}+\frac{|z|^2}{24x}\right)\\
&\ll x^{1/4}\exp\left(O(x^{1/4})+\frac{y^2}{24 x}\right)\\
&\ll x^{1/4}\exp\left(\frac{\pi^2}{6x}\cdot \frac{1}{4}\right),
\end{align*}
for all $x\in(0,1]$. On the other hand, using \eqref{eqdek} implies
$$\frac{1}{(e^{-z};e^{-z})_{\infty}}\ll |z|^{1/2}\exp\left(\frac{\pi^2 x}{6(x^2+y^2)}\right)\ll x^{1/4}\exp\left(\frac{\pi^2}{6x}\frac{1}{1+(y/x)^2}\right),$$
for all $|y|\ll x^{1/2}\le 1$. Combing the above we complete the proof of the lemma.
\end{proof}

We study the asymptotics of $H_{k,m,j}(q)$.  It is clear that $H_{k,m,j}(q)$ is a difference between two false theta functions. Recall that a partial theta function has the form of:
\begin{equation}\label{eqtf}
T_{a, b}(q)=\sum_{n\ge 1}(-1)^{n-1}q^{an^2+bn},
\end{equation}
where $a>0$, $b\in\br$ and $q=e^{-z}$ with $\Re(z)>0$.  The false/partial theta functions have recently appeared in several areas of mathematics, such as the theory of $q$-series, integer partitions and quantum topology. In all of these aspects, it is important to understand their asymptotic behavior. In the recent work, Liu and the author \cite[Theorem 2.7]{LZ} proved a uniform asymptotic involving $T_{a,b}(q)$. Throughout the paper, let $\P_{\alpha}=\frac{\,d}{\,d\alpha}$ denote the usual derivative operator. Then, their result \cite[Theorem 2.7]{LZ} is as follows.
\begin{theorem}\label{thmft}Let $p, \ell\in\bn_0$ be given. Also let $z=x+\ri y$ with $x,y\in\br$, $x>0$ and $|y|\le x$.  We have an asymptotic expansion
$$\sum_{n\ge 1}(-1)^{n-1}n^{\ell}e^{-n^2z-bnz}\sim (-1)^{\ell}\sum_{h\ge 0}\frac{(-z)^{h}}{h !}
\P_{w}^{2h+\ell}\bigg|_{w=bz}\left\{\frac{1}{1+e^{w}}\right\}\;\;\text{as} \;\;z\rrw 0,$$
with respect to the asymptotic sequence $\left(z^he^{-bz}\right)_{h\ge 0}$, uniformly in the parameter $b\ge 0$.
\end{theorem}
\begin{remark}
Before Liu and the author \cite[Theorem 2.7]{LZ}, the tools to obtain the asymptotics for \eqref{eqtf} are the Euler--Maclaurin summation formula and Mellin transform. We reference Zagier \cite{MR2257528} and Bringmann et al. recent work \cite{MR3210725, MR3597015, BJM20} on the Euler--Maclaurin summation formula, and Berndt and Kim \cite{MR3103192} and Mao \cite{MR2823024} on the Mellin transform. However, the above literature just deals with the asymptotic expansion of \eqref{eqtf} with $a, b$ fixed as $z\rrw 0$.
\end{remark}
We also need the following rough bound involving the theta functions.
\begin{lemma}\label{lemtheta}Let $p\ge 0$ and $u\in(0,1]$. We have
$$\sum_{\ell\in \bz} |\ell|^p e^{-\ell^2 u}\ll_p  u^{-(p+1)/2}.$$
\end{lemma}
\begin{proof}We first have
$$\sum_{\ell\in \bz} |\ell|^p e^{-\ell^2 u}\le 1+2\sum_{\ell\ge 1} \ell^p e^{-\ell^2 u}.$$
Using Euler-Maclaurin summation formula, we obtain
\begin{align*}
\sum_{\ell\ge 1} \ell^p e^{-\ell^2 u}&\le 1+\int_{0}^{\infty}\left(t^p e^{-t^2 u}+\left|\P_t(t^p e^{-t^2 u})\right| \right)\rd t\\
&=1+\int_{0}^{\infty}\left(u^{-\frac{p+1}{2}} t^p e^{-t^2}+u^{-\frac{p}{2}}\left|\P_t(t^p e^{-t^2})\right| \right)\rd t\\
&\ll_p u^{-(p+1)/2},
\end{align*}
this completes the proof.
\end{proof}
From the Theorem \ref{thmft} and Lemma \ref{lemtheta}, we obtain the following result.
\begin{proposition}\label{pro1}Let $m,j\in\bz$ such that $m+j\ge 0$ and $p\in\bn$ be given. Also let $z=x+\ri y$ with $x,y\in\br$, $x>0$ and $|y|\le x$.
We have an asymptotic expansion
$$H_{k,m,j}(e^{-z})\sim \sum_{\ell\ge 1}P_{k;\ell}^{(j)}(z)\P_{w}^{\ell}\bigg|_{w=mz}\left\{\frac{1}{1+e^{w}}\right\}\;\;\text{as} \;\;z\rrw 0,$$
with respect to the asymptotic sequence $\left(z^{\lceil(\ell+1)/2\rceil}e^{-mz}\right)_{\ell\ge 1}$, uniformly in the parameter $m\ge 0$. Here $P_{k;\ell}^{(j)}(z)$ is polynomial of degree $\ell$ defined by
$$P_{k;\ell}^{(j)}(z)=\sum_{\substack{h\ge 1, s\ge 0\\ h+2s=\ell}}\frac{c_{h}(j)}{h!}\frac{(k-1/2)^{s}}{s !} (-z)^{s+h},$$
with $c_\ell(j)=\left({1}/{2}-j\right)^\ell-(-1)^{\ell}\left({1}/{2}+j\right)^\ell$,
for each $\ell\in \bn$.
\end{proposition}
\begin{proof}
Using Taylor's theorem, we find that
$$(e^{w/2}-e^{-w/2})e^{-jw}=\sum_{1\le \ell\le p}\frac{c_\ell(j)}{\ell !}w^\ell+\frac{1}{p !}\int_{0}^w(w-u)^{p}\P_u^{p+1}\left((e^{u/2}-e^{-u/2})e^{-ju}\right)\rd u,$$
where $c_\ell(j)=\left({1}/{2}-j\right)^\ell-(-1)^{\ell}\left({1}/{2}+j\right)^\ell$.
This implies that
$$(e^{nz/2}-e^{-nz/2})e^{-jnz}-\sum_{1\le \ell\le p}\frac{c_\ell(j)}{\ell !}(nz)^\ell\ll_{p,j} (nz)^{p+1}e^{n(1/2+|j|)z}.$$
Inserting the above into \eqref{eqm210} we obtain
\begin{align}\label{eqm200}
H_{k,m,j}(e^{-z})&=\sum_{n\ge 1}(-1)^{n-1}e^{-(k-1/2)n^2z-mnz}\sum_{1\le \ell\le p}\frac{c_\ell(j)}{\ell !}(nz)^\ell\nonumber\\
&\quad +O_{p,j}\left(e^{-mz}\sum_{n\ge 1}|nz|^{p+1}e^{-(k-1/2)n^2x-(n-1)mx+n(1/2+|j|)x}\right)\nonumber\\
&=\sum_{1\le \ell\le p}\frac{c_\ell(j)}{\ell !}z^\ell\sum_{n\ge 1}(-1)^{n-1}n^{\ell}e^{-(k-1/2)n^2z-mnz}+O_{p,j}\left(|z|^{p/2}e^{-mz}\right).
\end{align}
Here we used the following fact that
\begin{align*}
\sum_{n\ge 1}n^{p}e^{-(k-\frac{1}{2})n^2x-(n-1)mx+n(\frac{1}{2}+|j|)x}&\ll \sum_{n\ge 1}|(n-1/2-|j|)+1/2+j|^{p}e^{-(n-\frac{1}{2}-|j|)^2(x/2)+(\frac{1}{2}+|j|)^2x/2}\\
&\ll_{p,j} \sum_{0\le r\le p}\sum_{n\ge 1}|n-1/2-|j||^{r}e^{-(n-\frac{1}{2}-|j|)^2(x/2)}\\
&\ll_{p,j} \sum_{\ell\in\bz}(1+|\ell|^p)e^{-\ell^2 x/4}\\
&\ll_{p} x^{-(p+1)/2}\ll |z|^{-(p+1)/2},
\end{align*}
which follows from Lemma \ref{lemtheta}. Therefore, by note that \eqref{eqm200} holds for each $p\ge 1$, the use of Theorem \ref{thmft} implies that
\begin{align*}
H_{k,m,j}(e^{-z})&\sim \sum_{\ell\ge 1}\frac{c_\ell(j)}{\ell !}(-z)^\ell \sum_{v\ge 0}\frac{(-(k-1/2)z)^{v}}{v !}
\P_{w}^{2v+\ell}\bigg|_{w=mz}\left\{\frac{1}{1+e^{w}}\right\}\\
&=\sum_{\ell \ge 1}\left(\sum_{\substack{h\ge 1, s\ge 0\\ h+2s=\ell}}\frac{c_{h}(j)}{h!}\frac{(k-1/2)^{s}}{s !} (-z)^{s+h}\right)\P_{w}^{\ell}\bigg|_{w=mz}\left\{\frac{1}{1+e^{w}}\right\}.
\end{align*}
This completes the proof.
\end{proof}

\section{Uniform asymptotic expansion for $N_k(m,n)$}\label{sec4}
In this section, we use Wright's version of the Circle Method to prove an asymptotic expansion of $N_k(m+j,n)$, from which we obtain the asymptotic log-concavity of $N_k(m,n)$ in next section.  We assume that $j\in\bz$ is fixed, and $m+j\ge 0$ such that $ m^2\beta_n^{3}=o(1)$ as $n\rrw+\infty$. The main result of this section is stated in the following.
\begin{theorem}\label{main0}We have
\begin{align}\label{eqasymp}
N_k(m+j,n)\sim \frac{\sqrt{3}\beta_n^{2}e^{\pi^2/3\beta_n}}{2\pi^2}\sum_{h\ge 1}\frac{c_{h}(j)}{h!}{\rm J}_{k, h}(m,n)(-\beta_n)^{h}\;\; \text{as}~ n\rrw \infty,
\end{align}
with respect to the asymptotic sequence $\left(\beta_n^he^{-m\beta_n}\right)_{h\ge 1}$ uniformly in
$m^2\beta_n^{3}=o(1)$. Here for each $h\in\bn$, ${\rm J}_{k, h}(m,n)$ has the following asymptotic expansion
\begin{equation*}
{\rm J}_{k, h}(m,n)\sim \sum_{r\ge 0} \frac{(-\beta_n)^{r}}{r!}\Upsilon_{h,r}\left(k;w, \P_w\right)\bigg|_{w=m\beta_n}\left\{\frac{1}{1+e^{w}}\right\},\;\; \text{as}~ n\rrw \infty,
\end{equation*}
with respect to the asymptotic sequence $\left(\beta_n^r(1+m^2\beta_n^{2})^re^{-m\beta_n}\right)_{r\ge 0}$ uniformly in $ m^2\beta_n^{3}=o(1)$. Here $\Upsilon_{h,0}=\P_w^h$, and for each $h\in\bn$ and $r\in\bn$,
$$\Upsilon_{h,r}\left(k; w, \P_w\right)=r!\sum_{\substack{s,\ell\ge 0\\ s+\ell=r}}\frac{(k-1/2)^{s}}{s!} \Upsilon_{h,s,\ell}\left(w, \P_w\right),$$
and for each $\ell\in\bn_0$, with $\gamma_\ell(\mu, v)$ defined by Proposition \ref{probh}, we have
$$\Upsilon_{h,s,\ell}\left(w, \P_w\right)=\left(\frac{3}{\pi^2}\right)^{\ell}\sum_{0\le v\le 2\ell}\gamma_\ell\left(-s-h-\frac{3}{2}, v\right)\frac{w^{v}}{v!}\P_{w}^{v+h+2s}.$$
\end{theorem}

We next present a direct consequences of Theorem \ref{main0}.   Notice that for each integer $\ell\ge 0$, by Proposition \ref{pro1} we have $c_{2\ell+1}(0)={1}/{4^{\ell}}$, $c_{2\ell}(0)=0$, and  $c_{\ell}(1)=(-1)^{\ell-1}(3^\ell-1)/{2^{\ell}}$.
\begin{corollary}\label{cor41}Uniformly for $m^2\beta_n^{3}=o(1)$,
$$N_k(m,n)=-\frac{\sqrt{3}\beta_n^{3}e^{\pi^2/3\beta_n}}{2\pi^2}\left({\rm J}_{k, 1}(m,n)-\frac{\beta_n^2}{24}{\rm J}_{k, 3}(m,n)+O\left(\beta_n^4e^{-m\beta_n}\right)\right)$$
and
$$N_k(m,n)-N_k(m+1,n)=\frac{\sqrt{3}\beta_n^{4}e^{\pi^2/3\beta_n}}{2\pi^2}\left({\rm J}_{k, 2}(m,n)-\frac{\beta_n}{2}{\rm J}_{k, 3}(m,n)+O\left(\beta_n^2e^{-m\beta_n}\right)\right),$$
as $n\rrw \infty$.
\end{corollary}

\subsection{Wright's version of the circle method}
We now prove Theorem \ref{main0}. We first recall and denote that
$$\beta_n=\pi/\sqrt{6(n-1/24)},\;\Lambda_n=\pi\sqrt{(n-1/24)/6},$$
with integer $n>0$ sufficiently large. Let $z=\beta_n+\ri y$ and use $\mathscr{C}$ to denote the circle on which $|q|=e^{-\beta_n}$.  By \eqref{eqmm0} and using Cauchy's theorem,
\begin{align*}
N_k(m+j,n)&=\frac{1}{2\pi \ri}\int_{\mathscr{C}}\mathcal{N}_{k,m+j}(q)\frac{\rd q}{q^{n+1}}\\
&=\frac{1}{2\pi }\int_{-\pi}^{\pi}\frac{1}{(e^{-z};e^{-z})_{\infty}}H_{k,m,j}(e^{-z})e^{nz}\rd  y\\
&=:M+E,
\end{align*}
where
\begin{align}\label{eqmfm}
M=\frac{1}{2\pi}\int_{|y|\le \beta_n}\frac{1}{(e^{-z};e^{-z})_{\infty}}H_{k,m,j}(e^{-z})e^{nz}\rd y,
\end{align}
and
\begin{align}\label{eqmfe}
E=\frac{1}{2\pi}\int_{\beta_n< |y|\le \pi}\frac{1}{(e^{-z};e^{-z})_{\infty}}H_{k,m,j}(e^{-z})e^{nz}\rd y.
\end{align}
We will show that the main asymptotic contribution comes from $M$.

\subsection{The estimate of $E$}
Using the definition of $H_{k,m,j}(e^{-z})$, the rough upper bounds Lemma \ref{lemtheta} for the theta functions,  we find for all $y\in\br$ that
\begin{equation}\label{eqth}
H_{k,m,j}(e^{-z})\ll_{k,j} \sum_{\ell\ge 1}e^{-((k-1/2)\ell^2+m\ell)\Re(z)}\ll_k \sum_{\ell\in \bz}e^{-\ell^2\beta_n/2}\ll \frac{1}{\sqrt{\beta_n}},
\end{equation}
as $n\rrw+\infty$. Recall Lemma \ref{lemdek} that
\begin{equation*}
\frac{1}{\left( e^{-z} ; e^{-z}\right)_\infty} \ll x^{1/4}e^{\pi^2/12x}\ll \beta_n^{1/4}e^{\pi^2/12\beta_n},
\end{equation*}
for all $x\le |y|\le \pi$. Then, from the above and the definition \eqref{eqmfe} of $E$,  we obtain
$$
E=\frac{1}{2\pi}\int_{ \beta_n<|y|\le \pi}\frac{H_{k,m,j}(z)e^{nz}}{(e^{-z};e^{-z})_{\infty}}\rd y\ll_{k,j} \beta_n^{-1/2}e^{\frac{\pi^2}{12\beta_n}+n\beta_n}\ll \beta_n^{-1/2}\exp\left(\frac{3}{2}\Lambda_n\right),
$$
as $n\rrw+\infty$. We conclude the above as the following proposition.
\begin{proposition}\label{proee} We have
$$E\ll_{j,k} \beta_n^{-1/2}\exp\left(\frac{3}{2}\Lambda_n\right),$$
as $n\rrw+\infty$.
\end{proposition}

\subsection{The estimate of $M$ and the proof of Theorem \ref{main0}}
The goal of this subsection is to determine the asymptotic expansions of $M$. Recall that the definition \eqref{eqmfm} of $M$, we have
\begin{align*}
M&=\frac{1}{2\pi}\int_{|y|\le \beta_n}\frac{H_{k,m,j}(e^{-z})e^{nz}}{(e^{-z};e^{-z})_{\infty}}\rd y\\
&=M_1+E_1,
\end{align*}
where
\begin{align*}
M_1:=\frac{1}{2\pi}\int_{|y|\le \beta_n}\frac{z^{1/2}}{\sqrt{2\pi}}e^{-\frac{z}{24}+\frac{\pi^2}{6z}}H_{k,m,j}(e^{-z})e^{nz}\rd y,
\end{align*}
and
\begin{align*}
E_1:=\frac{1}{2\pi}\int_{|y|\le \beta_n}\left(\frac{1}{(e^{-z};e^{-z})_{\infty}}-\frac{z^{1/2}}{\sqrt{2\pi}}e^{-\frac{z}{24}+\frac{\pi^2}{6z}}\right)H_{k,m,j}(e^{-z})e^{nz}\rd y.
\end{align*}
Using the rough upper bound \eqref{eqth} for $H_{k,m,j}(e^{-z})$,  and the asymptotics of $1/(e^{-z};e^{-z})_{\infty}$ in Lemma \ref{lemdek}, we obtain
\begin{align}\label{eqestme1}
E_1\ll \frac{1}{2\pi}\int_{|y|\le \beta_n}\left|H_{k,m,j}(e^{-z})e^{nz}\right|\rd y\ll \frac{1}{\sqrt{\beta_n}}e^{n\beta_n}\ll \exp\left(\frac{3}{2}\Lambda_n\right).
\end{align}
Using the uniform asymptotics of $H_{k,m,j}(e^{-z})$ in Proposition \ref{pro1},
\begin{align}\label{eqm1l}
M_1&=\frac{1}{(2\pi)^{3/2}}\int_{|y|\le \beta_n}e^{\frac{\pi^2}{6z}+(n-\frac{1}{24})z}z^{1/2}H_{k,m,j}(e^{-z})\rd y\nonumber\\
&\sim\frac{1}{(2\pi)^{3/2}}\int_{|y|\le \beta_n}e^{\frac{\pi^2}{6z}+(n-\frac{1}{24})z}z^{1/2}\left(\sum_{\ell\ge 1}P_{k;\ell}^{(j)}(z)\P_{w}^{\ell}\bigg|_{w=mz}\left\{\frac{1}{1+e^{w}}\right\}\right)\rd y\nonumber\\
&=:\frac{1}{\sqrt{2\pi}}\sum_{\ell \ge 1}M_{1,\ell},
\end{align}
where
\begin{align*}
M_{1,\ell}=\frac{1}{2\pi}\int_{|y|\le \beta_n}e^{\frac{\pi^2}{6z}+(n-\frac{1}{24})z}z^{1/2}P_{k;\ell}^{(j)}(z)\P_{w}^{\ell}\bigg|_{w=mz}\left\{\frac{1}{1+e^{w}}\right\}\rd y.
\end{align*}
Note that $P_{k;\ell}^{(j)}(z)\ll_{j,k} |z|^{\lceil(\ell+1)/2\rceil}$ and
$$\P_{w}^{p}\Big|_{w=mz}\left\{\frac{1}{1+e^{w}}\right\}\ll_p e^{-m\beta_n},$$
for all $|y|\le \beta_n$, we have
\begin{align*}
M_{1,\ell}&\ll  \int_{|y|\le \beta_n}e^{\frac{\pi^2\beta_n}{6(\beta_n^2+y^2)}+(n-\frac{1}{24})\beta_n}\beta_n^{1/2+\lceil(\ell+1)/2\rceil}e^{-m\beta_n}\rd y\\
&=\beta_n^{1/2+\lceil(\ell+1)/2\rceil}e^{2\Lambda_n-m\beta_n}\int_{|y|\le \beta_n}e^{-\frac{\pi^2y^2}{6\beta_n(\beta_n^2+y^2)}}\rd y\ll \beta_n^{2+\lceil(\ell+1)/2\rceil}e^{2\Lambda_n-m\beta_n}.
\end{align*}
On the other hand,
$$\int_{|y|\le \beta_n}e^{-\frac{\pi^2y^2}{6\beta(\beta_n^2+y^2)}}\rd y\le \int_{|y|\le \beta_n}e^{-\frac{\pi^2y^2}{12\beta_n^3}}\rd y=\beta_n^{3/2}\int_{|y|\le \beta_n^{-1/2}}e^{-\frac{\pi^2y^2}{12}}\rd y\ll \beta_n^{3/2}.$$
Hence we obtain
\begin{align*}
M_{1,\ell}\ll \beta_n^{2+\lceil(\ell+1)/2\rceil}e^{2\Lambda_n-m\beta_n}.
\end{align*}
Moreover, using Taylor's theorem for the smooth function $f(w)=\P_{w}^{\ell}\left(\frac{1}{1+e^{m\beta_n+w}}\right)$,
$$f(\ri my)=\sum_{0\le v<p}\frac{f^{(v)}(0)}{\ell !}(\ri my)^v+\frac{(\ri my)^p}{(p-1)!}\int_{0}^{1}(1-s)^{p-1}f^{(p)}(s\ri my)\,ds,$$
for any $p\ge 1$, and by noting that
$$\int_{0}^{1}(1-s)^{p-1}f^{(p)}(s\ri my)\,ds\ll \sup_{0\le s\le 1}|f^{(p)}(s\ri my)|\ll_{p,\ell} e^{-m\beta_n},$$
for $-\beta_n\le y\le \beta_n$, we have
$$\P_{w}^{\ell}\bigg|_{w=mz}\left\{\frac{1}{1+e^{w}}\right\}=\sum_{0\le v<p }\frac{\P_{w}^{\ell+v}\big|_{w=m\beta_n}\left\{\frac{1}{1+e^{w}}\right\}}{v!}(m\ri y)^v+O_p(|my|^pe^{-m\beta_n}).$$
Then, the above implies that
\begin{align*}
M_{1,\ell}&-\frac{1}{2\pi}\int_{|y|\le \beta_n}e^{\frac{\pi^2}{6z}+(n-\frac{1}{24})z}z^{1/2}P_{k;\ell}^{(j)}(z)\sum_{0\le v<p}\frac{\P_{w}^{\ell+v}\big|_{w=m\beta_n}\left\{\frac{1}{1+e^{w}}\right\}}{v!}(m\ri y)^v\rd y\\
&\ll_{j,k,p} \int_{|y|\le \beta_n}e^{\frac{\pi^2\beta_n}{6(\beta_n+y^2)}+(n-\frac{1}{24})\beta_n}\beta_n^{\frac{1}{2}+\lceil\frac{\ell+1}{2}\rceil}|my|^pe^{-m\beta_n}\rd y\\
&=m^p\beta_n^{1/2+\lceil(\ell+1)/2\rceil}e^{2\Lambda_n-m\beta_n}\int_{|y|\le \beta_n}|y|^pe^{-\frac{\pi^2y^2}{6\beta_n(\beta_n^2+y^2)}}\rd y\\
&\ll_{j,k,p} \beta_n^{2+\lceil\frac{\ell+1}{2}\rceil}|m\beta_n^{3/2}|^{p}e^{2\Lambda_n-m\beta_n}.
\end{align*}
Here we used that
$$\int_{|y|\le \beta_n}|y|^pe^{-\frac{\pi^2y^2}{6\beta_n(\beta_n^2+y^2)}}\rd y\le \int_{|y|\le \beta_n}|y|^pe^{-\frac{\pi^2y^2}{12\beta_n^3}}\rd y=\beta_n^{3(p+1)/2}\int_{|y|\le \beta_n^{-1/2}}|y|^pe^{-\frac{\pi^2y^2}{12}}\rd y\ll_p \beta_n^{3(p+1)/2}.$$
Making a change of variables in the above we obtain
\begin{align*}
M_{1,\ell}&-\sum_{0\le v< p}\frac{\P_{w}^{\ell+v}\big|_{w=m\beta_n}\left\{\frac{1}{1+e^{w}}\right\}}{v!}(m\beta_n)^v\beta_n^{3/2}{\mathscr I}_{v}\left(P_{k,\ell}^{(j)}\right)\ll_{j,k,p} \beta_n^{2+\lceil(\ell+1)/2\rceil}|m\beta_n^{3/2}|^{p}e^{2\Lambda_n-m\beta_n},
\end{align*}
where
$${\mathscr I}_{v}\left(P_{k,\ell}^{(j)}\right)=\frac{1}{2\pi\ri }\int_{1-\ri }^{1+\ri }e^{(\frac{1}{z}+z)\Lambda_n}z^{1/2}(z-1)^vP_{k;\ell}^{(j)}(\beta_n z)\rd z.$$
For given $v\in\bz_{\ge 0}$ and $\ell\in\bn$, using Lemma \ref{eqbesseq1} implies that
\begin{align*}
{\mathscr I}_{v}\left(P_{k,\ell}^{(j)}\right)=\frac{1}{2\pi\ri }\int_{\mathcal{H}}e^{(\frac{1}{z}+z)\Lambda_n}z^{1/2}(z-1)^vP_{k;\ell}^{(j)}(\beta_n z)\rd z+O_{j,k,\ell,v}(e^{3\Lambda_n/2}),
\end{align*}
by note that the integrand of ${\mathscr I}_{v}\left(P_{k,\ell}^{(j)}\right)$ is a linear combinations of certain finite integrands of the modified Bessel functions of first kind. Using the definition of $P_{k,\ell}^{(j)}$ in Proposition \ref{pro1}, and the definition \eqref{eqhmv} of $H_{\mu, v}$, the integral above can be evaluate as
\begin{align}\label{eqmmmm1}
\frac{1}{2\pi\ri }\int_{\mathcal{H}}&e^{(\frac{1}{z}+z)\Lambda_n}z^{1/2}(z-1)^vP_{k;\ell}^{(j)}(\beta_n z)\rd z\nonumber\\
&=\sum_{\substack{h\ge 1, s\ge 0\\ h+2s=\ell}}\frac{c_{h}(j)}{h!}\frac{(k-1/2)^{s}}{s !} (-\beta_n)^{s+h}\frac{1}{2\pi\ri }\int_{\mathcal{H}}e^{(\frac{1}{z}+z)\Lambda_n}z^{s+h+1/2}(z-1)^v\rd z\nonumber\\
&=\sum_{\substack{h\ge 1, s\ge 0,\\ h+2s=\ell}}\frac{c_{h}(j)}{h!}\frac{(k-1/2)^{s}}{s !} (-\beta_n)^{s+h} H_{-s-h-3/2, v}\left(2\Lambda_n\right).
\end{align}
Combining the above, for each $p\in\bn$ we obtain that
\begin{align*}
M_1=&\frac{1}{\sqrt{2\pi}}\sum_{1\le \ell<2p}M_{1,\ell}+O_{j,k,p}\left(\beta_n^{2+p}e^{2\Lambda_n-m\beta_n}\right)\\
=&\frac{1}{\sqrt{2\pi}}\sum_{1\le \ell<2p}\sum_{0\le v< 2p}\frac{\P_{w}^{\ell+v}\big|_{w=m\beta_n}\left\{\frac{1}{1+e^{w}}\right\}}{v!}(m\beta_n)^v\beta_n^{3/2}{\mathscr I}_{v}\left(P_{k,\ell}^{(j)}\right)+O_{j,k,p}\left(\beta_n^{3}|m^2\beta_n^{3}|^{p}e^{2\Lambda_n-m\beta_n}\right),
\end{align*}
Therefore, using the estimate \eqref{eqestme1} for $E_1$, the estimate for ${\mathscr I}_{v}\left(P_{k,\ell}^{(j)}\right)$, \eqref{eqmmmm1} and the definition that $\widehat{H}_{\mu, v}(u)=\sqrt{2\pi u}e^{-u}H_{\mu, v}(u)$, we have
\begin{align*}
M=&\frac{\sqrt{3}\beta_n^{2}e^{2\Lambda_n}}{2\pi^2}\sum_{\substack{1\le \ell<2p\\ 0\le v< 2p}}\frac{(m\beta_n)^v\P_{w}^{\ell+v}\big|_{w=m\beta_n}\left\{\frac{1}{1+e^{w}}\right\}}{v!}\sum_{\substack{h\ge 1, s\ge 0\\ h+2s=\ell}}\frac{(-\beta_n)^{s+h}c_{h}(j)(k-\frac{1}{2})^{s}}{h!s!} \widehat{H}_{-s-h-\frac{3}{2}, v}\left(2\Lambda_n\right)\\
&+O_{j,k,p}\left(\beta_n^{3}|m^2\beta_n^{3}|^{p}e^{2\Lambda_n-m\beta_n}\right).
\end{align*}
The above can be rewritten as
\begin{align*}
M=&\frac{\sqrt{3}\beta_n^{2}e^{2\Lambda_n}}{2\pi^2}\sum_{1\le h<2p}\frac{c_h(j)}{h!}{\rm J}_{k,h}^*(m,n)(-\beta_n)^h+O_{j,k,p}\left(\beta_n^{3}|m^2\beta_n^{3}|^{p}e^{2\Lambda_n-m\beta_n}\right),
\end{align*}
where
\begin{align*}
{\rm J}_{k,h}^*(m,n)=\sum_{0\le s<p}\frac{(-\beta_n)^{s}(k-{1}/{2})^{s}}{s!} \sum_{0\le v< 2p}\frac{(m\beta_n)^v\P_{w}^{h+2s+v}\big|_{w=m\beta_n}\left\{\frac{1}{1+e^{w}}\right\}}{v!} \widehat{H}_{-s-h-\frac{3}{2}, v}\left(\frac{\pi^2}{3\beta_n}\right).
\end{align*}
We denote the inner sum of the above by ${\rm J}_{k,h,s}^*(m,n)$, and by inserting the asymptotic expansion of $\widehat{H}_{\mu, v}$, that is Proposition \ref{probh}, we obtain
\begin{align*}
{\rm J}_{k,h,s}^*(m,n)\sim &\sum_{\ell\ge 0}(-\beta_n)^{\ell}\left(\frac{3}{\pi^2}\right)^{\ell}\sum_{0\le v< \min(2p, 2\ell+1)}\gamma_{\ell}\left(-s-h-\frac{3}{2}, v\right)\frac{(m\beta_n)^v\P_{w}^{h+2s+v}\big|_{w=m\beta_n}\left\{\frac{1}{1+e^{w}}\right\}}{v!}\\
=&\sum_{0\le \ell<p}(-\beta_n)^{\ell}\Upsilon_{h, s,\ell}(k; w, \P_w)\bigg|_{w=m\beta_n}\left\{\frac{1}{1+e^{w}}\right\}+O_{h,k,p}\left(\beta_n^p(1+m^2\beta_n^2)^pe^{-m\beta_n}\right).
\end{align*}
Here we used that
$$\Upsilon_{h, s,\ell}(k; w, \P_w)=\left(\frac{3}{\pi^2}\right)^{\ell}\sum_{0\le v\le 2\ell}\gamma_{\ell}\left(-s-h-\frac{3}{2}, v\right)\frac{w^v}{v!}\P_{w}^{v+h+2s}.$$
Inserting the above into the expression for ${\rm J}_{k,h}^*(m,n)$ of the above, we find that
\begin{align*}
{\rm J}_{k,h}^*(m,n)
&=\sum_{\substack{0\le s<p\\ 0\le \ell<p}}\frac{(-\beta_n)^{s+\ell}(k-\frac{1}{2})^{s}}{s!}\Upsilon_{h, s,\ell}(k; w, \P_w)\bigg|_{w=m\beta_n}\left\{\frac{1}{1+e^{w}}\right\}+O_{h,k,p}\left(\frac{\beta_n^p(1+m^2\beta_n^2)^p}{e^{m\beta_n}}\right)\\
&=\sum_{0\le r<p}\frac{(-\beta_n)^r}{r!}\Upsilon_{h, r}(k; w, \P_w)\bigg|_{w=m\beta_n}\left\{\frac{1}{1+e^{w}}\right\}+O_{h,k,p}\left(\beta_n^p(1+m^2\beta_n^2)^pe^{-m\beta_n}\right).
\end{align*}
Here we using
$$\Upsilon_{h, r}(k; w, \P_w)=r!\sum_{\substack{s,\ell\ge 0\\ s+\ell=r}}\frac{(k-1/2)^{s}}{s!} \Upsilon_{h,s,\ell}\left(w, \P_w\right).$$
Combing the estimates of $E$ in Proposition \ref{proee} and the above estimates for $M$, we completes the proof of the Theorem \ref{main0}.

\section{Log-concavity of $k$-rank functions}\label{sec5}
In this section we prove Theorem \ref{thmlca0}. We first prove the case of $|m|\le \beta_n^{\varepsilon-3/2}$ with any fixed small $\varepsilon>0$.
\subsection{The case of $|m|\le \beta_n^{\varepsilon-3/2}$}\label{subsec51}
Let us denote as
$$L_{k}(m,n)=\frac{N_k(m,n)^2-N_k(m-1,n)N_k(m+1,n)}{(2\pi^2)^{-2}\cdot 3\beta_n^{6}e^{4\Lambda_n}}.$$

\medskip

Using Theorem \ref{main0} we prove the following asymptotic formula, which can be used to determine the sign of $L_k(m, n)$ when $n$ tends to infinity.
\begin{proposition}\label{lem51}Uniformly for $ 0<m^2\beta_n^{3}=o(1)$,
$$L_{k}(m,n)=\left({\rm J}_{k, 2}(m,n)^2-{\rm J}_{k, 1}(m,n){\rm J}_{k, 3}(m,n)\right)\beta_n^{2}+O(e^{-2m\beta_n}\beta_n^4),~~ \text{as}~~ n\rrw \infty.$$
\end{proposition}

\begin{remark}
We note that $4(N_k(m,n)^2-N_k(m-1,n)N_k(m+1,n))$ is the discriminant of the Jensen polynomial
$${\frak J}_{k,n}^{2, m}(X):=\binom{2}{0}N_k(m-1,n)+\binom{2}{1}N_k(m,n)X+\binom{2}{2}N_k(m+1,n)X^2,$$
of degree $2$ and shift $m-1$ of the sequence $(N_k(m,n))_{m\in\bz}$.
It is seems that the sign of $L_k(m, n)$ for large enough $n$ can be determined by rewriting the asymptotic expansion of $N_k(m+j,n)$ in Theorem \ref{main0} as a formal power series of $j$, and then using Griffin--Ono--Rolen--Zagier \cite[Theorem 6]{MR3963874}.
However, our Theorem \ref{thmlca0} is a bivariate uniform asymptotic, the coefficients of the associated Jensen polynomial after using \cite[Theorem 6]{MR3963874} (suitable variable scaling) are still complicated functions related to $m$ and $n$. Therefore, the hyperbolicity (or log-concavity) of ${\frak J}_{k,n}^{2, m}(X)$ cannot be directly given by using \cite[Theorem 6]{MR3963874}.
\end{remark}

\begin{proof}Notice that $c_{\ell}(j)=(1/2-j)^\ell-(-1)^{\ell}(1/2+j)^{\ell}$, then using Theorem \ref{main0} we obtain
\begin{align*}
L_{k}(m,n)&\sim\sum_{\ell\ge 2}\sum_{\substack{h_1,h_2\ge 1\\ h_1+h_2=\ell}}\frac{\beta_n^{\ell-2}{\rm J}_{k, h_1}(m,n){\rm J}_{k, h_2}(m,n) }{h_1!h_2!}\left(c_{h_1}(0)c_{h_2}(0)-c_{h_1}(-1)c_{h_2}(1)\right)\\
&=\sum_{\ell\ge 1}\sum_{\substack{h_1,h_2\ge 1\\ h_1+h_2=2\ell\\ h_1\equiv h_2\equiv 1\pmod 2}}\frac{(\beta_n/2)^{2\ell-2}{\rm J}_{k, h_1}(m,n){\rm J}_{k, h_2}(m,n) }{h_1!h_2!}\left(1-\frac{(3^{h_1}-1)(3^{h_2}-1)}{4}\right)\\
&\quad+\sum_{\ell\ge 1}\sum_{\substack{h_1,h_2\ge 1\\ h_1+h_2=2\ell\\ h_1\equiv h_2\equiv 0\pmod 2}}\frac{(\beta_n/2)^{2\ell-2}{\rm J}_{k, h_1}(m,n){\rm J}_{k, h_2}(m,n)  }{h_1!h_2!}\frac{(3^{h_1}-1)(3^{h_2}-1)}{4}.
\end{align*}
Since ${\rm J}_{k, h}(m,n)\ll_{k,h} e^{-m\beta_n}$, we find that
\begin{align*}
L_{k}(m,n)&\sim\sum_{\ell\ge 0}\left(\frac{\beta_n}{2}\right)^{2\ell+2}\sum_{\substack{h_1,h_2\ge 0\\ h_1+h_2=\ell+1}}\frac{{\rm J}_{k, 2h_1+1}(m,n){\rm J}_{k, 2h_2+1}(m,n) }{(2h_1+1)!(2h_2+1)!}\left(1-\frac{(3^{2h_1+1}-1)(3^{2h_2+1}-1)}{4}\right)\\
&\quad+\sum_{\ell\ge 0}\left(\frac{\beta_n}{2}\right)^{2\ell+2}\sum_{\substack{h_1,h_2\ge 0\\ h_1+h_2=\ell\\ }}\frac{{\rm J}_{k, 2h_1+2}(m,n){\rm J}_{k, 2h_2+2}(m,n) }{(2h_1+2)!(2h_2+2)!}\frac{(3^{2h_1+2}-1)(3^{2h_2+2}-1)}{4}\\
&=\left({\rm J}_{k, 2}(m,n)^2-{\rm J}_{k, 1}(m,n){\rm J}_{k, 3}(m,n) \right)\beta_n^{2}+O_k(e^{-2m\beta_n}\beta_n^4),
\end{align*}
which completes the proof.
\end{proof}
Using the definition of ${\rm J}_{k, h}(m,n)$ in Theorem \ref{main0}, we further prove the following lemma.
\begin{lemma}\label{lem52}Let $p\in\bn$. Uniformly for $ 0<m\beta_n^{3/2}=o(1)$,
\begin{align*}
{\rm J}_{k, 2}(m,n)^2-&{\rm J}_{k, 1}(m,n){\rm J}_{k, 3}(m,n)\\
=&\frac{1+O_{k,p}(\beta_n+m^2\beta_n^3)}{32}{\rm sech}^6\left(\frac{m\beta_n}{2}\right)\\
&+\frac{3\beta_n(1+O_{k,p}(\beta_n+m^2\beta_n^3+m^{2p}\beta_n^{3p-1}+e^{-m\beta_n}))}{16\pi^2}{\rm sech}^4\left(\frac{m\beta_n}{2}\right),
\end{align*}
as $n\rrw \infty$.
\end{lemma}
\begin{proof}
Recall from Theorem \ref{main0} that
\begin{equation}\label{eqjj}
{\rm J}_{k, h}(m,n)=\Upsilon_{h,0}+\sum_{1\le r<p}(-\beta_n)^r\Upsilon_{h,r}+O_{p,h,k}\left(\beta_n^p(1+(m\beta_n)^{2p})e^{-m\beta_n}\right),
\end{equation}
where
$$\Upsilon_{h,r}=\frac{1}{r!}\Upsilon_{h,r}\left(k;w, \P_w\right)\bigg|_{w=m\beta_n}\left\{\frac{1}{1+e^{w}}\right\}$$
with $\Upsilon_{h,0}=\P_w^h$, and for each $h\in\bn$ and $r\in\bn$,
$$\Upsilon_{h,r}\left(k; w, \P_w\right)=r!\sum_{\substack{s,\ell\ge 0\\ s+\ell=r}}\frac{(k-1/2)^{s}}{s!} \left(\frac{3}{\pi^2}\right)^{\ell}\sum_{0\le v\le 2\ell}\gamma_\ell\left(-s-h-\frac{3}{2}, v\right)\frac{w^{v}}{v!}\P_{w}^{v+h+2s},$$
where $\gamma_\ell(\mu, v)$ be defined by Proposition \ref{probh}. Furthermore, since $\gamma_\ell\left(-s-h-\frac{3}{2}, v\right)$ is polynomial of $(h-2)$ with degree $2\ell-v$, and
$$\P_{w}^{h}\bigg|_{w=m\beta_n}\left\{\frac{1}{1+e^{w}}\right\}=(-1)^{h}e^{-m\beta_n}\left(1+O_h(e^{-m\beta_n})\right),$$
it is easy to find that there exists a polynomial $P_{k, 2r-\ell}$ of degree at most $2r-\ell$ such that
\begin{align}\label{eqpp}
\Upsilon_{h,r}=&(-1)^he^{-m\beta_n}\sum_{0\le \ell\le 2r}P_{k, 2r-\ell}(m\beta_n)(h-2)^{\ell}+O_{h,r,k}\left((1+(m\beta_n)^{2r})e^{-2m\beta_n}\right).
\end{align}
In particular, $P_{k, 0}={3}/{2\pi^2}$ is a constant, by noting that $\gamma_1(\mu, 0)=\mu^2/2-1/8$. Inserting \eqref{eqpp} into \eqref{eqjj} and noting that $m^2\beta_n^3=o(1)$, it is easy to find that
\begin{align*}
{\rm J}_{k, h}(m,n)=&\Upsilon_{h,0}+(-1)^he^{-m\beta_n}\sum_{0\le \ell<2p}(h-2)^{\ell}Q_\ell+O_{k,h,p}\left(\frac{1+(m\beta_n)^{2}}{\beta_n^{-1}e^{2m\beta_n}}+\frac{\beta_n^p(1+(m\beta_n)^{2p})}{e^{m\beta_n}}\right),
\end{align*}
where
$$Q_\ell:=Q_\ell(\beta_n, m\beta_n)=\sum_{\max(1,\ell/2)\le r<p}(-\beta_n)^rP_{k, 2r-\ell}(m\beta_n).$$
Therefore, by noting that
$$Q_\ell\ll_{k,h,p,\ell} \sum_{\max(1,\ell/2)\le r<p}\beta_n^r |m\beta_n|^{2r-\ell}\ll_{k,h,p} \beta_n^{\max(1,\ell/2)}|m\beta_n|^{\max(2-\ell,0)},$$
we have
\begin{align*}
{\rm J}_{k, 2}(m,n)^2-&{\rm J}_{k, 1}(m,n){\rm J}_{k, 3}(m,n)\\
=&\left(\Upsilon_{2,0}+e^{-m\beta_n}Q_0\right)^2+O_{k,p}\left(\beta_n(1+(m\beta_n)^{2})e^{-3m\beta_n}+\beta_n^p(1+(m\beta_n)^{2p})e^{-2m\beta_n}\right)\\
&-\bigg(\Upsilon_{1,0}-e^{-m\beta_n}\sum_{0\le \ell<2p}(-1)^{\ell}Q_\ell\bigg)\bigg(\Upsilon_{3,0}-e^{-m\beta_n}\sum_{0\le \ell<2p}Q_\ell\bigg)\\
=&\Upsilon_{2,0}^2-\Upsilon_{1,0}\Upsilon_{3,0}+e^{-2m\beta_n}\left(2Q_0-\sum_{0\le \ell<2p}Q_\ell-\sum_{0\le \ell<2p}(-1)^\ell Q_\ell\right)\\
&+\frac{1}{e^{2m\beta_n}}\bigg(Q_0^2-\sum_{0\le \ell<2p}(-1)^{\ell}Q_\ell\sum_{0\le \ell<2p}Q_\ell\bigg)+O_{k,p}\left(\frac{\beta_n+m^2\beta_n^{3}}{e^{3m\beta_n}}+\frac{\beta_n^p(1+(m\beta_n)^{2p})}{e^{2m\beta_n}}\right).
\end{align*}
Further simplification implies that
\begin{align*}
{\rm J}_{k, 2}(m,n)^2-&{\rm J}_{k, 1}(m,n){\rm J}_{k, 3}(m,n)\\
&=\Upsilon_{2,0}^2-\Upsilon_{1,0}\Upsilon_{3,0}-2e^{-2m\beta_n}\left(Q_2+O_{k,p}(\beta_n^2)\right)\\
&\quad+\frac{1}{e^{2m\beta_n}}\bigg(O_p(\beta_n^2(1+m^2\beta_n^2))\bigg)+O_{k,p}\left(\frac{\beta_n+m^2\beta_n^{3}}{e^{3m\beta_n}}+\frac{\beta_n^p(1+(m\beta_n)^{2p})}{e^{2m\beta_n}}\right)\\
&=\Upsilon_{2,0}^2-\Upsilon_{1,0}\Upsilon_{3,0}+\frac{3\beta_ne^{-2m\beta_n}}{\pi^2}+O_{k,p}\left(\frac{\beta_n+m^2\beta_n^{3}}{e^{3m\beta_n}}+\frac{\beta_n^2+m^2\beta_n^4+(m^2\beta_n^3)^{p}}{e^{2m\beta_n}}\right).
\end{align*}
Here we used the fact that $Q_2=-\beta_nP_{k, 0}(m\beta_n)+O_{k,p}(\beta_n^2(1+(m\beta_n)^2))$ and $P_{k, 0}(m\beta_n)=3/2\pi^2$. Noting that
$$\Upsilon_{2,0}^2-\Upsilon_{1,0}\Upsilon_{3,0}=\frac{1}{32}{\rm sech}^6\left(\frac{m\beta_n}{2}\right)\;\;\text{and}\;\;e^{-2m\beta_n}=\frac{1+O(e^{-m\beta_n})}{16}{\rm sech}^4\left(\frac{m\beta_n}{2}\right),$$
we obtain
\begin{align*}
{\rm J}_{k, 2}(m,n)^2-&{\rm J}_{k, 1}(m,n){\rm J}_{k, 3}(m,n)\\
=&\frac{1+O_{k,p}(\beta_n+m^2\beta_n^3)}{32}{\rm sech}^6\left(\frac{m\beta_n}{2}\right)\\
&+\frac{3\beta_n(1+O_{k,p}(\beta_n+m^2\beta_n^3+m^{2p}\beta_n^{3p-1}+e^{-m\beta_n}))}{16\pi^2}{\rm sech}^4\left(\frac{m\beta_n}{2}\right).
\end{align*}
This completes the proof of the lemma.
\end{proof}
Combining Proposition \ref{lem51} and Lemma \ref{lem52} we obtain asymptotic log-concavity of the sequence $\left(N_k(m,n)\right)_{m\in\bn}$. We now prove the case of $m=0$. From Corollary \ref{cor41}, it is not difficult to find that
$$
\frac{N_k(0,n)}{p(n)}=\frac{\beta_n}{4}\left(1+O_k\left(\beta_n\right)\right),\;\frac{N_k(0,n)+N_k(1,n)}{p(n)}=\frac{\beta_n}{2}\left(1+O_k\left(\beta_n\right)\right),
$$
and
$$
\frac{N_k(0,n)-N_k(1,n)}{p(n)}
=\frac{\beta_n^3}{16}\left(1+O_k\left(\beta_n\right)\right),
$$
as $n\rrw +\infty$. This immediately implies the following:
\begin{align}\label{eqref11}
1-\frac{N_k(1,n)^2}{N_k(0,n)^2}=\frac{\beta_n^2}{2}\left(1+O_k\left(\beta_n\right)\right).
\end{align}
Then, combining \eqref{eqref11} and $N_k(-m,n)=N_k(m,n)$, the use of Proposition \ref{lem51} and Lemma \ref{lem52} will implies the following:
\begin{align}\label{eqref13}
1-\frac{N_k(m-1,n)N_k(m+1,n)}{N_k(m,n)^2}=\frac{\beta_n^2}{2}{\rm sech}^2\left(\frac{m\beta_n}{2}\right)
+\frac{3\beta_n^3}{\pi^2}+O_{k,\varepsilon}\left(m^2\beta_n^5(\beta_n+e^{-m\beta_n})\right),
\end{align}
for all $m\in\bz$ such $|m|\le \beta_n^{\varepsilon-3/2}$ with any fixed small $\varepsilon>0$. Noting that if the sequence $\left(f(\ell)\right)_\ell$ such that ${f(\ell+1)}/{f(\ell)}\rrw 1$ as $\ell\rrw+\infty$, then
\begin{align}\label{eqref12}
\Delta_\ell^2\log f(\ell)&=\log\left(1-\left(1-\frac{f(\ell-1)f(\ell+1)}{f(\ell)^2}\right)\right)\nonumber\\
&=\frac{f(\ell)^2-f(\ell-1)f(\ell+1)}{f(\ell)^2}+O\left(\left(\frac{f(\ell)^2-f(\ell-1)f(\ell+1)}{f(\ell)^2}\right)^2\right), \end{align}
and \eqref{eqref13}, we immediately obtain the following theorem.
\begin{theorem}\label{lnkmmae} Uniformly for $m\in\bz$ such that $|m|\le \beta_n^{\varepsilon-3/2}$,
$$-\Delta_m^2\log N_k(m,n)=\frac{\beta_n^2}{2}{\rm sech}^2\left(\frac{m\beta_n}{2}\right)+\frac{3\beta_{n}^3}{\pi^2}+O_{k,\varepsilon}\left(m^2\beta_n^5(\beta_n+e^{-m\beta_n})\right),$$
as $n\rrw+\infty$, where $\varepsilon>0$ is any fixed small real number.
\end{theorem}
\subsection{The cases of $3\log n\le |m|\beta_n\le n/2$}\label{subsec52}
In this subsection, we prove the log-concavity of $k$-rank functions $N_k(m,n)$ with condition ${|m|\beta_n}\ge 3{\log n}$. We prove that
\begin{theorem}\label{thm41} For all $m\in\bz$ such that ${|m|\beta_n}\ge 3{\log n}$ and $n-|m|\rrw+\infty$,
$$-\Delta_m^2\log N_k(m,n)=\frac{3}{\pi^2}\beta_{n-|m|}^3+O\left(\beta_{n-|m|}^4\right).$$
\end{theorem}
To prove this theorem, we first note that
\begin{equation*}
N_{k}(m,n)=\sum_{\ell\ge 1}(-1)^{\ell-1}\left(p\left(n-\frac{(2k-1)\ell^2-\ell}{2}-|m|\ell\right)-p\left(n-\frac{(2k-1)\ell^2-\ell}{2}-(|m|+1)\ell\right)\right)
\end{equation*}
by using \eqref{eqm0}, then for $n\ge |m|+k+1$ we have
\begin{align}\label{eqm1}
N_{k}(m,n)=&p\left(\ell+1\right)-p\left(\ell\right)+O\left(\sqrt{\ell}p\left(\ell-|m|-3k+3\right)\right),
\end{align}
where $\ell=n-|m|-k$. Note the well-known facts that
$$p(\ell+1)-p(\ell)\gg \ell^{-1/2}p(\ell)\;\mbox{and}\; p(\ell-r)\ll p(\ell)e^{-r\beta_\ell},$$
uniformly for all $0\le r< \ell/2$.
We have if $\ell\rrw+\infty$ then
\begin{equation*}
N_k(m,n)= \left(p\left(\ell+1\right)-p\left(\ell\right)\right)\left(1+O\left(\ell e^{-|m|\beta_\ell}\right)\right).
\end{equation*}
This implies that
$$\frac{N_{k}(m-1,n)N_{k}(m+1,n)}{N_{k}(m,n)^2}=\frac{\left(p\left(\ell\right)-p\left(\ell-1\right)\right)\left(p\left(\ell+2\right)-p\left(\ell+1\right)\right)}{\left(p\left(\ell+1\right)-p\left(\ell\right)\right)^2}\left(1+O\left(\ell e^{-|m|\beta_\ell}\right)\right),$$
as $\ell\rrw+\infty$. In view of this, we first prove the following Proposition \ref{pro52}, and by noting that
$$\ell e^{-|m|\beta_\ell}\ll n e^{-|m|\beta_n}= n^{1-\frac{|m|\beta_n}{\log n}}\ll n^{-2}\ll \beta_\ell^4,$$
for all $m$ such that $|m|\beta_n\ge 3\log n$, then Theorem \ref{thm41} follows from a straightforward calculation.
\begin{proposition}\label{pro52}We have for $n\in\bn$ with $n\rrw \infty$,
\begin{align*}
-\Delta_n^2\log(p(n+1)-p(n))=\frac{3\beta_n^{3}}{\pi^2}+O\left(\beta_n^4\right).
\end{align*}
In other words the sequence $\left(p(n+1)-p(n)\right)_{n\ge 1}$ is eventually log-concave.
\end{proposition}
For this purpose, we need the following Hardy--Ramanujan asymptotic result for $p(n)$, which is an immediate consequence of \cite[Eq.(1.61)]{MR1575586}.
\begin{lemma}\label{lem1}We have for $n\in\bn$, with $w=n-1/24\rrw \infty$,
\begin{equation*}
p(n)=\frac{1}{\sqrt{3}B^2}\P_{w}^2 e^{B\sqrt{w}}+O(e^{B\sqrt{w}/2}),
\end{equation*}
where $B=2\pi/\sqrt{6}$.
\end{lemma}
\begin{remark}
Theoretically, when we notice that $p(n)$ has an asymptotic expansion
$$p(n)\sim \frac{1}{4\sqrt{3}n}e^{2\pi \sqrt{n/6}}\left(1+\frac{c_1}{n^{1/2}}+\frac{c_2}{n}+\cdots\right),\; (n\rrw \infty)$$
where $c_1,c_2,\ldots$ are some constants, or other more precise asymptotic formulas of $p(n)$, the proof idea of Proposition \ref{pro52} is direct and there is no obstacle. However, here we want to give a proof that the number of estimates (the number of inequality scaling) is as few as possible. In this sense, the above lemma is necessary.
\end{remark}

We define
$$B(w):=\P_{w}^2 (e^{B\sqrt{w+1}}-e^{B\sqrt{w}}),$$
and prove the following lemma.
\begin{lemma}\label{lem42} We have
$$\P_w^2\log B(w)=-\frac{B}{4w^{3/2}}+O\left(\frac{1}{w^2}\right).$$
\end{lemma}
\begin{proof}Using the definition $B(w)$ gives that
\begin{align*}
B(w)=\P_w^2\left(e^{B\sqrt{w}}\left(\exp\left(\frac{B}{\sqrt{w}+\sqrt{w+1}}\right)-1\right)\right).
\end{align*}
After taking a derivative, we find that there exists an analytic function $G_B\left(u\right)$ defined on a neighborhood of $0$, with $G_B(0)>0$ such that
\begin{align*}
B(w)=w^{-3/2}e^{B\sqrt{w}}G_B\left(\frac{1}{\sqrt{w}}\right).
\end{align*}
Therefore, as $w\rrw +\infty$,
\begin{align*}
\log B(w)=B\sqrt{w}-\frac{3}{2}\log w+\log G_B(0)+\sum_{j\ge 1}\frac{d_j(B)}{w^{j/2}}.
\end{align*}
Hence
$$\P_w^2\log B(w)=-\frac{B}{4w^{3/2}}-\frac{3}{2w^2}+\sum_{j\ge 1}\frac{j(j+2)d_j(B)}{4w^{2+j/2}},$$
which completes the proof.
\end{proof}

\begin{proof}[The proof of Proposition \ref{pro52}]
We now establish the asymptotic log-concavity of sequence $$\left(p(n+1)-p(n)\right)_{n\ge 1}.$$
Using above Lemma \ref{lem1} and perform straightforward calculation, we obtain
\begin{align*}
&\frac{(p(n+2)-p(n+1))(p(n)-p(n-1))}{(p(n+1)-p(n))^2}\\
&\qquad\qquad=\frac{\P_{w}^2 \left(e^{B\sqrt{w+2}}-e^{B\sqrt{w+1}}\right)\P_{w}^2 \left(e^{B\sqrt{w}}-e^{B\sqrt{w-1}}\right)}{\left(\P_{w}^2 (e^{B\sqrt{w+1}}-e^{B\sqrt{w}})\right)^2}+O\left(e^{-Bw^{1/3}}\right)\\
&\qquad\qquad=\frac{\left(\int_{0}^1\P_{w}^3e^{B\sqrt{w+1+t}}\rd  t\right)\left(\int_{0}^1\P_{w}^3e^{B\sqrt{w-1+t}}\rd  t\right)}{\left(\int_{0}^1\P_{w}^3e^{B\sqrt{w+t}}\rd  t\right)^2}+O(e^{-Bw^{1/3}}).
\end{align*}
Inserting the fact that
$$
e^{B\sqrt{w+t\pm 1}}=e^{B\sqrt{w+t}}\pm\int_{0}^{1}\P_we^{B\sqrt{w+t\pm v}}\rd v
$$
into above and perform straightforward calculation, we obtain
\begin{align*}
&1-\frac{(p(n+2)-p(n+1))(p(n)-p(n-1))}{(p(n+1)-p(n))^2}\\
&\qquad=1-\left(1+\frac{1}{B(w)}\int_{0}^1\P_wB(w+v)\rd v\right)\left(1-\frac{1}{B(w)}\int_{0}^1\P_wB(w-v)\rd v\right)+O(e^{-Bw^{1/3}})\\
&\qquad=-\frac{\int_{0}^1\rd v\int_{-v}^v\P_w^2B(w+\delta)\rd \delta}{B(w)}+\frac{\int_{0}^1\P_wB(w+v)\rd v\int_{0}^1\P_wB(w-v)\rd v}{B(w)^2}+O(e^{-Bw^{1/3}}).
\end{align*}
Noting that $\P_wB(w)\ll w^{-1/2}B(w)$ and
$$\P_w^rB(w+\eta)=\P_w^rB(w)+\eta \P_w^{r+1}B(w)+ O\left( w^{-\frac{r+2}{2}}B(w)\right),$$
for $\eta\in[-1,1]$ and for $r\ge 1$, we have
$$\int_{0}^1\rd v\int_{-v}^v\P_w^2B(w+\delta)\rd \delta-\P_w^2B(w)\ll w^{-2}B(w)$$
and
$$\int_{0}^1\P_wB(w+v)\rd v\int_{0}^1\P_wB(w-v)\rd v-\left(\P_wB(w)\right)^2\ll w^{-2}B(w)^2.$$
This implies that
\begin{align*}
&1-\frac{(p(n+2)-p(n+1))(p(n)-p(n-1))}{(p(n+1)-p(n))^2}\\
&\qquad\qquad=-\frac{\P_w^2B(w)}{B(w)}+\frac{(\P_wB(w))^2}{B(w)^2}+O(w^{-2})\\
&\qquad\qquad=-\P_w^2\log B(w)+O(w^{-2}).
\end{align*}
Using Lemma \ref{lem42} and and \eqref{eqref12} completes the proof of the proposition.
\end{proof}
\subsection{The cases of $|m|\ge n/2$}\label{subsec53}
In this subsection we establish the log-concavity of $N_k(m,n)$ with $m\ge n/2$. In particular, we prove Proposition \ref{prolcl0}. Recall Proposition \ref{prolcl0} state that
$$N_k(m,n)^2-N_k(m-1,n)N_k(m+1,n)<0,$$
for all $k\in\bn$ and $n/2-2k+2\le |m|\le n-k-71$.
By \eqref{eqm1}, we have if $|m|\ge n/2-2k+2$ then
\begin{align*}
N_{k}(m,n)=p\left(n-|m|-k+1\right)-p\left(n-|m|-k\right).
\end{align*}
Proposition \ref{prolcl0} follows from the following log-concavity result of
$\left(p(n+1)-p(n)\right)_{n\ge 0}$ easily.
\begin{proposition}\label{pro51}For all $\ell\ge 71$ we have
$$\left(p\left(\ell+1\right)-p\left(\ell\right)\right)^2-\left(p\left(\ell\right)-p\left(\ell-1\right)\right)\left(p\left(\ell+2\right)-p\left(\ell+1\right)\right)<0.$$
In other words the sequence
$\left(p(n+1)-p(n)\right)_{n\ge 71}$ is log-concave.
\end{proposition}
\begin{proof}
The cases of $1\le \ell\le 95$ can be checked by {\bf Mathematica}. We only prove the case for $\ell\ge 95$.
From the main theorem of Chen-Jia-Wang \cite{MR3976587}, we have
$$4(p(\ell)^2-p(\ell-1)p(\ell+1))(p(\ell+1)^2-p(\ell)p(\ell+2))>(p(\ell)p(\ell+1)-p(\ell-1)p(\ell+2))^2$$
for all $\ell\ge 95$. This implies
\begin{align*}
\left(p(\ell)^2-p(\ell-1)p(\ell+1)+p(\ell+1)^2-p(\ell)p(\ell+2)\right)^2>(p(\ell)p(\ell+1)-p(\ell-1)p(\ell+2))^2,
\end{align*}
by note that $(x+y)^2\ge 4xy$ for all $x,y\in \br$. Using the fact that the sequence $\left(p(\ell)\right)_{\ell\ge 0}$ is log-concave for all $\ell\ge 25$, and together with the above we obtain
\begin{align*}
p(\ell)^2-p(\ell-1)p(\ell+1)+p(\ell+1)^2-p(\ell)p(\ell+2)>p(\ell)p(\ell+1)-p(\ell-1)p(\ell+2).
\end{align*}
Through simple deformation, we find that the above is equivalent to Proposition \ref{pro51}. This completes the proof of the proposition.
\end{proof}




\bigskip
\noindent
{\sc Nian Hong Zhou\\
School of Mathematics and Statistics, Guangxi Normal University\\
No.1 Yanzhong Road, Yanshan District, Guilin, 541006\\
Guangxi, PR China}\newline
Email:~\href{mailto:nianhongzhou@outlook.com; nianhongzhou@gxnu.edu.cn}{\small nianhongzhou@outlook.com; nianhongzhou@gxnu.edu.cn}

\end{document}